\documentclass[12pt]{amsart}

\usepackage{tikz-cd}
\usepackage{amsfonts}
\usepackage{amsthm}
\usepackage{xfrac}
\usepackage{faktor}
\usepackage{amssymb}
\usepackage{enumitem}
\usepackage{microtype}
\usepackage{amsthm}
\usepackage{amsmath}
\usepackage{amssymb}
\usepackage{mathtools}
\usepackage{microtype}
\usepackage{calrsfs}
\DeclareMathAlphabet{\pazocal}{OMS}{zplm}{m}{n}
\usepackage{mathabx}
\usepackage{amsmath}
\usepackage[all]{xy}

\title[Finiteness criteria for Gorenstein homological dimension of groups]{Finiteness criteria for Gorenstein homological dimension and some invariants of groups}
\author{Ilias Kaperonis and Dimitra-Dionysia Stergiopoulou}
\thanks{Corresponding Author: Dimitra-Dionysia Stergiopoulou}

\newtheorem{Lemma}{Lemma}[section]
\newtheorem{Proposition}[Lemma]{Proposition}
\newtheorem{Theorem}[Lemma]{Theorem}
\newtheorem{Corollary}[Lemma]{Corollary}
\newtheorem{Remark}[Lemma]{Remark}
\newtheorem{Definition}[Lemma]{Definition}

\begin{document}

\begin{abstract} In this paper, we study finiteness criteria for the Gorenstein homological dimension of groups over a commutative ring of finite Gorenstein weak global dimension and provide estimates for the Gorenstein weak global dimension of group rings. As a result, we obtain Gorenstein analogues of well known properties in classical homological algebra over large families of infinite groups. Moreover, we prove that over a commutative ring of finite Gorenstein  weak global dimension, the Gorenstein cohomological dimension of a group $G$ bounds its Gorenstein homological dimension. Finally, we compare the generalized cohomological dimension and the generalized homological dimension of a group.
\end{abstract}

\maketitle
%\tableofcontents

\section{Introduction}
The concept of $G$-dimension for commutative Noetherian rings was introduced by Auslander and Bridger \cite{AB}; it was extended to modules over any ring through the notion of Gorenstein projective modules by Enochs and Jenda \cite{EJ,EJ2}.  Enochs and Jenda also defined the notions of Gorenstein injective and Gorenstein flat modules. The relative homological dimensions based on these modules were defined and studied in \cite{H1}, which is the standard reference for these notions. The Gorenstein homological dimension of a group $G$ over $\mathbb{Z}$ was defined in \cite{Asa} as the Gorenstein flat dimension of the trivial $\mathbb{Z}G$-module $\mathbb{Z}$ and yields a generalization of the classical homological dimension of a group $G$. Similarly, the Gorenstein homological dimension $\textrm{Ghd}_k G$ of a group $G$ over a commutative ring $k$ is defined as the Gorenstein flat dimension of the trivial $kG$-module $k$.

This paper considers a commutative ring $k$ of finite Gorenstein weak global dimension and proves finiteness criteria for the Gorenstein homological dimension of a group $G$ over $k$.  A central role in the characterization of the finiteness of $\textrm{Ghd}_k G$ is played by the notion of weak characteristic module for the group $G$ over $k$, which is introduced here (see Definition \ref{defi}). It generalizes the notion of characteristic module which was used in \cite{BDT, ET, Tal} to prove properties of the Gorenstein cohomological dimension and recently of the projectively coresolved Gorenstein dimension of groups \cite{St}. Part of our motivation comes from \cite[Section 2]{St}, where finiteness criteria for the projectively coresolved Gorenstein flat dimension and the Gorenstein cohomological dimension of groups are given.

Let us describe the main results of this paper in some detail. Many of them require the finiteness of the invariant $\textrm{sfli}k$ for the commutative ring $k$. This is defined as the supremum of the flat lengths (dimensions) of injective left $k$-modules; it generalizes the invariant $\textrm{spli}k$ defined by Gedrich and Gruenberg \cite{GG} as the supremum of the projective lengths (dimensions) of injective left $k$-modules. As a concequence of \cite[Theorem 2.4]{CET}, the finiteness of $\textrm{sfli}k$ is equivalent to the finiteness of the Gorenstein weak global dimension of $k$. The first main result of this paper compares the Gorenstein homological dimension and the Gorenstein cohomological dimension of a group $G$ over a commutative ring $k$ of finite Gorenstein weak global dimension (see Theorem \ref{prop312}).

\begin{Theorem}Let $k$ be a commutative ring such that $\textrm{sfli}k<\infty$ and $G$ be a group. Then, $\textrm{Ghd}_k G \leq \textrm{Gcd}_k G$.
\end{Theorem}

Our second main result provides finiteness criteria for the Gorenstein homological dimension of a group $G$ over a commutative ring $k$ of finite Gorenstein weak global dimension (see Theorem \ref{theo310}).

\begin{Theorem}Let $k$ be a commutative ring such that $\textrm{sfli}k<\infty$. Then, the following conditions on a group $G$ are equivalent:
	\begin{itemize}
		\item[(i)] $\textrm{Gfd}_{kG}M<\infty$ for every $kG$-module $M$.
		\item[(ii)] $\textrm{Ghd}_k G <\infty$.
		\item[(iii)] There exists a short exact sequence of $kG$-modules $0\rightarrow k \rightarrow A \rightarrow \overline{A} \rightarrow 0$, where $\textrm{fd}_{kG}A <\infty$ and the $kG$-module $\overline{A}$ is $PGF$.
		\item[(iv)]There exists a short exact sequence of $kG$-modules $0\rightarrow k \rightarrow A' \rightarrow \overline{A'} \rightarrow 0$, where $\textrm{fd}_{kG}A' <\infty$ and the $kG$-module $\overline{A'}$ is Gorenstein flat.
		\item[(v)] There exists a $k$-split $kG$-monomorphism $\iota: k \rightarrow A$, where $\textrm{fd}_{kG}A <\infty$ and the $kG$-module ${A}$ is $k$-projective.
		\item[(vi)] There exists a $k$-pure $kG$-monomorphism $\iota: k \rightarrow A'$, where $\textrm{fd}_{kG}A' <\infty$ and the $kG$-module ${A'}$ is $k$-flat.
		\item[(vii)] $\textrm{sfli}(kG) <\infty$.
	\end{itemize}
	Moreover, in this case, every Gorenstein projective $kG$-module is Gorenstein flat and hence we have $\textrm{Gfd}_{kG}M\leq\textrm{Gpd}_{kG}M$ for every $kG$-module $M$.
\end{Theorem}

Holm’s metatheorem \cite{Ho} claims that every result in classical homological algebra has a counterpart in Gorenstein homological algebra. We consider a commutative ring $k$ such that $\textrm{sfli}k<\infty$ and a group $G$ which admits a weak characteristic module. Then, Theorem 1.2 implies that every Gorenstein projective $kG$-module is Gorenstein flat. Moreover, under the above conditions we show that a $kG$-module is Gorenstein flat if and only if its Pontryagin dual module is Gorenstein injective (see Proposition \ref{prop45}). Since groups of type $\Phi_k$ and $\textsc{LH}\mathfrak{F}$ groups of type FP$_{\infty}$ admit weak characteristic modules, we obtain the following result (see Theorem \ref{theo54}). 

%The projectively coresolved Gorenstein flat modules were introduced by Saroch and Stovicek \cite{S-S}. Over a ring $R$, these modules are the syzygies of the acyclic complexes of projective modules that remain acyclic after applying the functor $I\otimes_R  \_\!\_$ for every injective module $I$. It is clear that every projectively coresolved Gorenstein flat module is Gorenstein flat. Moreover, every projectively coresolved Gorenstein flat module is Gorenstein projective (see \cite[Theorem 4.4]{S-S}).

\begin{Theorem}Let $k$ be a commutative ring such that $\textrm{sfli}k<\infty$ and $G$ be a group of type $\Phi_k$ or an $\textsc{LH}\mathfrak{F}$ group of type FP$_{\infty}$. Then:
	\begin{itemize}
		\item[(i)]A $kG$-module is Gorenstein projective if and only if it is projectively coresolved Gorenstein flat.
		\item[(ii)]A $kG$-module is Gorenstein flat if and only if its Pontryagin dual module is Gorenstein injective.
	\end{itemize}
\end{Theorem}

While for every commutative ring $k$ and every group $G$ the Gorenstein cohomological dimension ${\textrm{Gcd}}_kG$ vanishes if and only if $G$ is a finite group (see \cite[Corollary 2.3]{ET}), this is not true for the Gorenstein homological dimension (see Corollary \ref{theo62}).

\begin{Theorem}The following conditions on a group $G$ are equivalent:
	\begin{itemize}
		\item[(i)] $G$ is locally finite.
		\item[(ii)]$\textrm{Ghd}_k G=0$ for every commutative ring $k$.
		\item[(iii)]$\textrm{Ghd}_{\mathbb{Z}} G=0$.
	\end{itemize}
\end{Theorem}

Furthermore, we examine the relation between the generalized homological and the generalized cohomological dimension of groups defined by Ikenaga \cite{In} (see Corollary \ref{prop16}).

\begin{Theorem}We have $\underline{\textrm{hd}}(G)\leq \underline{\textrm{cd}}(G)$ for every group $G$.
\end{Theorem}

Our next main result yields an analogue of Serre’s Theorem for the Gorenstein homological dimension of groups (see Theorem \ref{theo612}).

\begin{Theorem}Let $k$ be a commutative ring such that $\textrm{sfli}k<\infty$. Consider a group $G$ and a subgroup $H$ of $G$ of finite index. Then, $\textrm{Ghd}_{k}G=\textrm{Ghd}_{k}H$.
\end{Theorem}

This paper also provides approximations for the invariant $\textrm{sfli}(kG)$ and the Gorenstein weak global dimension of $kG$ (see Corollary \ref{cor57} and Corollary \ref{cor58}).

\begin{Theorem}Let $k$ be a commutative ring and $G$ be a group. Then,
	\begin{itemize}
		\item[(i)] $\textrm{max}\{\textrm{Ghd}_{k}G,\textrm{sfli}k\}\leq \textrm{sfli}(kG)\leq \textrm{Ghd}_{k}G+\textrm{sfli}k,$
		\item[(ii)] $\textrm{max}\{\textrm{Ghd}_{k}G,\textrm{Gwgl.dim}k\}\leq \textrm{Gwgl.dim}(kG)\leq \textrm{Ghd}_{k}G+\textrm{Gwgl.dim}k.$
	\end{itemize} 
\end{Theorem}

Finally, subadditivity results for the invariant $\textrm{sfli}(kG)$ and the Gorenstein weak global dimension of $kG$ are obtained (see Proposition \ref{prop65} and Corollary \ref{cor76}).

\begin{Theorem}Let $k$ be a commutative ring and consider a group $G$, a normal subgroup $H$ of $G$ and the corresponding quotient group $Q=G/H$. Then,
\begin{itemize}
	\item[(i)]$\textrm{sfli}(kG) \leq \textrm{sfli}(kH) + \textrm{Ghd}_k Q \leq \textrm{sfli}(kH) + \textrm{sfli}(kQ),$
	\item[(ii)]$\textrm{Gwgl.dim}(kG) \leq \textrm{Gwgl.dim}(kH) + \textrm{Ghd}_k Q\leq \textrm{Gwgl.dim}(kH) + \textrm{Gwgl.dim}(kQ).$
\end{itemize}	
\end{Theorem}

 The contents of the paper are as follows. Section 2 establishes notation, terminology and preliminary results that will be used in the sequel. Section 3 introduces the notion of weak characteristic module for a group $G$ over a commutative ring $k$ and provides finiteness criteria for the Gorenstein homological dimension of $G$ over $k$, when $k$ has finite Gorenstein weak global dimension (see Theorem \ref{theo310}). Using the notion of weak characteristic module, the relationship between the Gorenstein homological dimension and the Gorenstein cohomological dimension of a group $G$ over a commutative ring $k$ of finite Gorenstein weak global dimension is studied (see Theorem \ref{prop312}). The existence of a weak characteristic module over a commutative ring of finite Gorenstein weak global dimension provides Gorenstein analogues of well known properties of modules in classical homological algebra. Under the above conditions, it is proven that (a) every Gorenstein projective $kG$-module is Gorenstein flat and (b) a $kG$-module is Gorenstein flat if and only if its Pontryagin dual is Gorenstein injective. Applications are given in Section 4, where groups of type $\Phi_k$ and $\textsc{LH}\mathfrak{F}$-groups of type FP$_{\infty}$ are considered (see Theorem \ref{theo54}).
 
 Following Ikenaga \cite{In}, Section 5 defines the generalized homological dimension of groups over any commutative ring $k$ and studies the relationship between this invariant and the Gorenstein homological dimension of groups over any commutative ring of finite Gorenstein weak global dimension (see Proposition \ref{prop316}). It is deduced that for every commutative ring $k$ and every group $G$, the Gorenstein homological dimension ${\textrm{Ghd}}_kG$ vanishes if and only if $G$ is a locally finite group (see Theorem \ref{theo62}). Moreover, it is proven that the generalized cohomological dimension bounds the generalized homological dimension over any commutative $\aleph_0$-Noetherian ring of finite global dimension (see Theorem \ref{propp16}). Finally, a Gorenstein homological analogue of Serre's theorem \cite[VIII, Theorem 3.1]{Br} is obtained (see Theorem \ref{theo612}).
 
 Section 6 provides estimates for the invariant $\textrm{sfli}(kG)$ in terms of the dimension $\textrm{Ghd}_{k}G$ and the invariant $\textrm{sfli}k$ (see Corollary \ref{cor57}). It is also proven there that the invariant $\textrm{sfli}(kG)$ is subadditive under group extensions over any commutative ring (see Proposition \ref{prop65}). As a result, estimates for the Gorenstein weak global dimension of $kG$ are obtained. 

\smallskip

\noindent\textbf{Conventions.} All rings are assumed to be associative and unital and all ring homomorphisms will be unit preserving. Unless otherwise specified, all modules will be left $R$-modules. We denote by $\textrm{Mod}(R)$ the category of $R$-modules.

\section{Preliminaries}In this  section, we establish notation, terminology and preliminary results that will be used in the sequel.

\subsection{Gorenstein modules and dimensions}
An acyclic complex $\textbf{P}$ of projective modules is said to be a complete 
projective resolution if the complex of abelian groups $\mbox{Hom}_R(\textbf{P},Q)$
is acyclic for every projective module $Q$. Then, a module is Gorenstein 
projective if it is a syzygy of a complete projective resolution. We denote by ${\tt GProj}(R)$ the class of Gorenstein projective $R$-modules. The 
Gorenstein projective dimension $\mbox{Gpd}_RM$ of a module $M$ is the 
length of a shortest resolution of $M$ by Gorenstein projective modules. 
If no such resolution of finite length exists, then we write 
$\mbox{Gpd}_RM = \infty$. If $M$ is a module of finite projective dimension, then $M$ has finite Gorenstein projective dimension as well and $\mbox{Gpd}_RM = \mbox{pd}_RM$ (see \cite{H1}).

An acyclic complex $\textbf{I}$ of injective modules is said to be a complete 
injective resolution if the complex of abelian groups $\mbox{Hom}_R(J,\textbf{I})$
is acyclic for every injective module $J$. Then, a module is Gorenstein 
injective if it is a syzygy of a complete injective resolution. We denote by ${\tt GInj}(R)$ the class of Gorenstein injective $R$-modules. The 
Gorenstein injective dimension $\mbox{Gid}_RM$ of a module $M$ is the 
length of a shortest (right) resolution of $M$ by Gorenstein injective modules. 
If no such resolution of finite length exists, then we write 
$\mbox{Gid}_RM = \infty$. If $M$ is a module of finite injective dimension, then $M$ has finite Gorenstein injective dimension as well and $\mbox{Gid}_RM = \mbox{id}_RM$ (see \cite{H1}).

An acyclic complex $\textbf{F}$ of flat modules is said to be a complete flat
resolution if the complex of abelian groups $I \otimes_R \textbf{F}$ is acyclic 
for every injective right module $I$. Then, a module is Gorenstein 
flat if it is a syzygy of a complete flat resolution. We denote by ${\tt GFlat}(R)$ the class of Gorenstein flat $R$-modules. The Gorenstein flat dimension 
$\mbox{Gfd}_RM$ of a module $M$ is the length of a shortest resolution 
of $M$ by Gorenstein flat modules. If no such resolution of finite length
exists, then we write $\mbox{Gfd}_RM = \infty$. If $M$ is a module of 
finite flat dimension, then $M$ has finite Gorenstein flat dimension as
well and $\mbox{Gfd}_RM = \mbox{fd}_RM$ (see \cite[Theorem 2.2]{Ben} or \cite[Lemma 2.4]{KS}). 

The notion of a projectively coresolved Gorenstein flat module (PGF-module, for short) is introduced by Saroch and Stovicek \cite{S-S}. A PGF module is a syzygy of an acyclic complex of projective modules $\textbf{P}$, which is such that the complex of abelian groups $I \otimes_R \textbf{P}$ is acyclic for every injective module $I$. We denote by ${\tt PGF}(R)$ the class of PGF $R$-modules. It is clear that the class ${\tt PGF}(R)$ is contained in ${\tt GFlat}(R)$. Moreover, we have the inclusion ${\tt PGF}(R) \subseteq {\tt GProj}(R)$ (see \cite[Theorem 4.4]{S-S}). Finally, the class of PGF $R$-modules is closed under extensions, direct sums, direct summands and kernels of epimorphisms. 

The notion of a Ding projective module is introduced by Ding, Li and Mao \cite{DLM}, where they were called strongly Gorenstein flat modules. Later, Gillespie \cite{Gi} renamed these modules Ding projective modules. A Ding projective module is a syzygy of an acyclic complex of projective modules $\textbf{P}$ which is such that the complex of abelian groups $\textrm{Hom}_R(\textbf{P},F)$ is acyclic for every flat module $F$. We denote by ${\tt DP}(R)$ the class of Ding projective $R$-modules. It is clear that the class ${\tt DP}(R)$ is contained in ${\tt GProj}(R)$. Moreover, it follows easily from \cite[Theorem 4.11]{S-S} that ${\tt PGF}(R)\subseteq {\tt DP}(R)$. Hence, we have the inclusions ${\tt PGF}(R)\subseteq {\tt DP}(R)\subseteq {\tt GProj}(R)$. 

\medskip\noindent \textbf{Cotorsion pairs.} Let $\mathcal{L}$ be a class of $R$-modules. We define its left orthogonal class as $^\perp\mathcal{L}=\{X\in R\text{-}\textrm{Mod}\mid \textrm{Ext}^1(X,L)=0\text{, }  \forall L\in\mathcal{L}\}$. The right orthogonal class of $\mathcal{L}$ is defined dually. An ordered pair of classes $(\mathcal{K},\mathcal{L})$ is called a cotorsion pair if $\mathcal{K}^\perp=\mathcal{L}$ and $K=$ $^\perp\mathcal{L}$. A cotorsion pair $(\mathcal{K},\mathcal{L})$ is called complete if for every $R$-module $M$ there are approximation sequences of the form \[ 0 \rightarrow L \rightarrow K \rightarrow M \rightarrow 0 \,\,\, \,\textrm{and}\,\,\,\, 0 \rightarrow M\rightarrow L' \rightarrow K' \rightarrow 0,\] where $K,K'\in\mathcal{K}$ and $L,L'\in\mathcal{L}$. Finally, the cotorsion pair $(\mathcal{K},\mathcal{L})$ is called hereditary if $ \textrm{Ext}^{i>0}(K,L)=0$ for every $L\in\mathcal{L}$ and every $K\in\mathcal{K}$. In that case, the class $\mathcal{K}$ is closed under kernels of epimorphisms, while the class $\mathcal{L}$ is closed under cokernels of monomorphisms.
%\begin{Lemma}\label{lemataki}{\rm(\cite[Lemma 2.4]{KS})} Let $M$ be a Gorenstein flat $R$-module of finite flat dimension. Then $M$ is flat.\end{Lemma}

%\begin{Corollary}\label{corolaki}{\rm(\cite[Corollary 2.5]{KS})} Let $M$ be an $R$-module of finite flat dimension. Then, $\textrm{Gfd}_R M =\textrm{fd}_R M$.\end{Corollary}

\begin{Lemma} \label{lem41}Let $M$ be a PGF $R$-module such that $\textrm{fd}_R M <\infty$. Then, $M$ is $R$-projective.
\end{Lemma}
\begin{proof}
	Let $\textrm{fd}_R M=n$. If $n=0$, then the PGF $R$-module is also flat. Since $M$ is PGF, there exists a short exact sequence of $R$-modules of the form $0\rightarrow M \rightarrow P \rightarrow N \rightarrow 0$, where the $R$-module $P$ is projective and the $R$-module $N$ is PGF as well. Since $M$ is $R$-flat, invoking \cite[Theorem 4.11]{S-S} we obtain that $\textrm{Ext}^1_{R}(N,M)=0$ and hence the exact sequence above is $R$-split. Thus, the $R$-module $M$ is projective as a direct summand of $P$. In the case where $n>0$, we consider a truncated projective resolution of $M$ of length $n$ $$0\rightarrow K_n \rightarrow P_{n-1} \rightarrow \cdots \rightarrow P_1 \rightarrow P_0 \rightarrow M \rightarrow 0.$$ Since $\textrm{fd}_R M =n$, the $R$-module $K_n$ is flat. Then, \cite[Proposition 2.1(iii)]{KS} and a simple inductive argument yield that the $R$-module $K_n$ is also PGF. Hence, the case $n=0$ implies that $K_n$ is projective and $\textrm{pd}_R M \leq n$. Since the $R$-module $M$ is $PGF$ and hence Gorenstein projective, we conclude that $M$ is projective.
\end{proof}

\begin{Lemma}\label{lemX}Let $\overline{{\tt GFlat}}_n(R)$ be the subcategory of $R$-modules having Gorenstein flat dimension less than $n\geq 0$. Then, $\overline{{\tt GFlat}}_n(R)$ is closed under direct limits.\end{Lemma}

\begin{proof}Consider a direct system of $R$-modules $\{M_i\}_{i\in I}$ such that $M_i\in \overline{{\tt GFlat}}_n(R)$ for every $i\in I$. Since $\textrm{Gfd}_{R}M_i\leq n$, there exists a truncated flat resolution of $M_i$ $$0\rightarrow K_{i,n}\rightarrow F_{i,n-1}\rightarrow \cdots \rightarrow F_{i,1}\rightarrow F_{i,0}\rightarrow M_i \rightarrow 0,$$ where $K_{i,n}$ is Gorenstein flat (see \cite[Corollary 4.12]{S-S} and \cite[Theorem 2.8]{Bennis}) for every $i\in I$. We obtain an exact sequence of $R$-modules of the form \begin{equation}\label{eq2224}0\rightarrow {\lim\limits_{\longrightarrow}}_i K_{i,n}\rightarrow {\lim\limits_{\longrightarrow}}_i F_{i,n-1}\rightarrow \cdots \rightarrow {\lim\limits_{\longrightarrow}}_i F_{i,1}\rightarrow {\lim\limits_{\longrightarrow}}_i F_{i,0}\rightarrow {\lim\limits_{\longrightarrow}}_i M_i \rightarrow 0,\end{equation} which constitutes a truncated flat resolution of ${\lim\limits_{\longrightarrow}}_i M_i$ of length $n$. Since the class ${\tt GFlat}(R)$ is closed under direct limits (see \cite[Corollary 4.12]{S-S}), we conclude that the $R$-module ${\lim\limits_{\longrightarrow}}_i K_{i,n}$ is Gorenstein flat. Invoking again \cite[Theorem 2.8]{Bennis} and the sequence (\ref{eq2224}), we infer that $\textrm{Gfd}_R ({\lim\limits_{\longrightarrow}}_i M_i)\leq n$, as needed.\end{proof}

%\begin{Lemma}\label{lemY}Let $k$ be a ring and consider an integer $n\geq 1$ and an exact sequence of $k$-modules $$0\rightarrow M_n \rightarrow \cdots \rightarrow M_1 \rightarrow M_0 \rightarrow M \rightarrow 0.$$ Then, $\textrm{Gfd}_{k}M\leq \textrm{max}\{\textrm{Gfd}_{k}M_i+i: i=0,\dots,n\}$.\end{Lemma}

%\begin{proof}We proceed by induction on $n\geq 1$. Since every ring is ${\tt GF}$-closed (see \cite[Corollary 4.12]{S-S}), the case $n=1$ follows from \cite[Theorem 2.11]{Bennis}. We assume now that $n>1$ and let $M'=\textrm{Im}(M_1 \rightarrow M_0)$. Applying the induction hypothesis on the exact sequence $$0\rightarrow M_n \rightarrow \cdots \rightarrow M_1 \rightarrow M' \rightarrow 0,$$ we obtain that $\textrm{Gfd}_{k}M'\leq \textrm{max}\{\textrm{Gfd}_{k}M_i+i-1: i=1,\dots,n\}$. Using again \cite[Theorem 2.11]{Bennis} and the short exact sequence $0\rightarrow M' \rightarrow M_0 \rightarrow M \rightarrow 0$, we conclude that $\textrm{Gfd}_{k}M\leq \textrm{max}\{\textrm{Gfd}_{k}M_0, \textrm{Gfd}_{k}M'+1\}\leq \textrm{max}\{\textrm{Gfd}_{k}M_i+i: i=0,1,\dots,n\}$. \end{proof}

We denote by $\overline{{\tt Flat}}(R)$ (resp. by $\overline{{\tt GFlat}}(R)$) the class of $R$-modules of finite flat dimension (resp. of finite Gorenstein flat dimension). Then, we have the following result.
\begin{Lemma}\label{lem22}$\overline{{\tt GFlat}}(R)\cap {\tt PGF}^{\perp}(R)=\overline{{\tt Flat}}(R)$.
\end{Lemma}

\begin{proof}The inclusion $\overline{{\tt Flat}}(R)\subseteq \overline{{\tt GFlat}}(R)\cap {\tt PGF}^{\perp}(R)$, follows easily from \cite[Theorem 2.2]{Ben} and \cite[Theorem 4.11]{S-S}. We consider now an $R$-module $M\in \overline{{\tt GFlat}}(R)\cap {\tt PGF}^{\perp}(R)$ and let $\textrm{Gfd}_R M=n<\infty$. We also consider a truncated flat resolution of $M$ of length $n$: $$0\rightarrow K_n \rightarrow F_{n-1}\rightarrow \cdots \rightarrow F_0 \rightarrow M \rightarrow 0,$$ where $K_n$ is Gorenstein flat (see \cite[Theorem 2.8]{Bennis}). Since $F_i\in {\tt PGF}^{\perp}(R)$ for every $i=0,\dots ,n-1$ (see \cite[Theorem 4.11]{S-S}), and the class ${\tt PGF}^{\perp}(R)$ is closed under kernels of epimorphisms (see \cite[Theorem 4.9]{S-S}), we conclude that $K_n$ is flat and hence $\textrm{fd}_R M\leq n <\infty$, as needed.
\end{proof}

\begin{Proposition}\label{Gfd}Let $R$ be a ring and $m$ a nonnegative integer, and consider an $R$-module $M$ such that $\textrm{Gfd}_R M=m$. Then, there exists a short exact sequence of $R$-modules of the form $0\rightarrow M \rightarrow A \rightarrow G \rightarrow 0$, where $\textrm{fd}_R A= m$ and $G$ is PGF.
\end{Proposition}

\begin{proof}Since $\mathfrak{PGF}=({\tt PGF},{\tt PGF}^{\perp})$ is a complete hereditary cotorsion pair (see \cite[Theorem 4.9]{S-S}), there exists a short exact sequence of $R$-modules of the form 
	\begin{equation}\label{eq24}0\rightarrow M \rightarrow A \rightarrow G \rightarrow 0,\end{equation} where $A\in {\tt PGF}^{\perp}(R)$ and $G\in {\tt PGF}(R)$. Since every ring is ${\tt GF}$-closed, invoking \cite[Theorem 2.11]{Bennis}, the exact sequence (\ref{eq24}) yields $\textrm{Gfd}_R A= m$. Consequently, the $R$-module $A$ has finite flat dimension by Lemma \ref{lem22}. In particular, we have  $\textrm{fd}_RA=\textrm{Gfd}_R A= m$ (see \cite[Theorem 2.2]{Ben}).
\end{proof}
\smallskip
For every ring $R$ and every left $R$-module $M$, we denote by $DM$ the Pontryagin dual $\textrm{Hom}_{\mathbb{Z}}(M, \mathbb{Q}/\mathbb{Z})$ of $M$.

The following result is due to Bouchiba \cite[Theorem 4]{Bo}. We provide a shorter proof.

\begin{Proposition}\label{Ginj}Let $R$ be a ring and consider a left $R$-module $M$ such that $\textrm{Gfd}_R M<\infty$. Then, $DM$ is Gorenstein injective if and only if $M$ is Gorenstein flat.\end{Proposition}
\begin{proof}It is known that for every Gorenstein flat module $M$, the Pontryagin dual $DM$ of $M$ is Gorenstein injective (see \cite[Theorem 3.6]{H1}). Consequently, it suffices to show that if $\textrm{Gfd}_R M<\infty$ and $DM$ is Gorenstein injective, then $M$ is Gorenstein flat. Let $I$ be an injective right $R$-module. Since $DM$ is Gorenstein injective, invoking \cite[Theorem 2.22]{H1} and the adjointness isomorphism $D \textrm{Tor}_i^R(I,M)\cong \textrm{Ext}^i_R(I,DM)$, $i\geq 0$, we infer that $D \textrm{Tor}_i^R(I,M)=0$ for every $i>0$. Therefore, we conclude that $\textrm{Tor}_i^R(I,M)=0$ for every $i>0$. Since $\textrm{Gfd}_R M<\infty$ and every commutative ring is GF-closed (see \cite[Corollary 4.12]{S-S}), invoking \cite[Theorem 2.8]{Bennis} we conclude that $M$ is Gorenstein flat, as needed.
\end{proof}

	%Lemma \ref{Gfd} yields the existence of a short exact sequence of left $R$-modules of the form 	\begin{equation}\label{27.2}0\rightarrow M \rightarrow A \rightarrow G \rightarrow 0,	\end{equation} where $\textrm{fd}_R A<\infty$ and $G$ is PGF. Since $G$ is PGF (and hence Gorenstein flat), it follows that $\textrm{Tor}_i^R(I,G)=0$, for every $i>0$ (see \cite[Lemma 2.4]{Bennis}). Thus, the exact sequence (\ref{27.2}) implies that $\textrm{Tor}_i^R(I,A)=\textrm{Tor}_i^R(I,M)=0$, for every injective right $R$-module $I$ and every $i>0$. Let $\textrm{fd}_R A=n<\infty$ and consider a right $R$-module $N$ and a truncated injective resolution of $N$ of length $n$:	\begin{equation}\label{27.3}0\rightarrow N \rightarrow I_0 \rightarrow \cdots \rightarrow I_{n-1}\rightarrow K_n \rightarrow 0.\end{equation} Then, the exact equation (\ref{27.3}) yields $\textrm{Tor}_i^R(N,A)=\textrm{Tor}_{i+n}^R(K,A)=0$ for every $i>0$ and hence the left $R$-module $A$ is flat. Invoking \cite[Theorem 4.11]{S-S} and the exact sequence (\ref{27.2}), we conclude that $M$ is Gorenstein flat, as needed.

\subsection{Gedrich-Gruenberg invariants and Gorenstein global dimensions}The invariants $\textrm{silp}R$, $\textrm{spli}R$ were defined by Gedrich and Gruenberg in \cite{GG} as the supremum of the injective lengths (dimensions) of projective modules and the supremum of the projective lengths (dimensions) of injective modules, respectively. Analogously, the invariant $\textrm{sfli}R$ is defined as the supremum of the flat lengths (dimensions) of injective left $R$-modules, while the invariant $\textrm{silf}R$ is defined as the supremum of the injective lengths (dimensions) of flat modules. Emmanouil and Talelli proved that $\textrm{silf}R=\textrm{silp}R$ for any ring $R$ (see \cite[Proposition 2.1]{ET2}). The Gorenstein weak global dimension $\textrm{Gwgl.dim}R$ of the ring $R$ is defined as $\textrm{Gwgl.dim}R=\{\textrm{Gfd}_R M: \, M\in \textrm{Mod}(R)\}$. The following result provides a finiteness criterion for the Gorenstein weak global dimension of a ring $R$.
\begin{Theorem}\label{theo012}{\rm (\cite[Theorem 2.4]{CET})} Let $R$ be a unital associative ring. The following conditions are equivalent:
	\begin{itemize}
		\item[(i)] $\textrm{Gwgl.dim}R<\infty$
		\item[(ii)] $\textrm{Gfd}_R M<\infty$ for every module $M$
		\item[(iii)] $\textrm{sfli}R<\infty$ and $\textrm{sfli}R^{\textrm{op}}<\infty$
	\end{itemize}
	Then, $\textrm{Gwgl.dim}R=\textrm{sfli}R=\textrm{sfli}R^{\textrm{op}}<\infty$.
\end{Theorem}
It follows that for every ring $R$ such that $R\cong R^{\textrm{op}}$, the finiteness of $\textrm{Gwgl.dim}R$ is equivalent to the finiteness of the invariant $\textrm{sfli}R$.
\subsection{Group rings}Let $k$ be a commutative ring, $G$ be a group and $R=kG$ be the associated group algebra. Then, we have an isomorphism $kG\cong {(kG)}^{\textrm{op}}$. Indeed, the anti-isomorphism of $kG$ which is induced by the map $g\rightarrow g^{-1}$, $g\in G$, enables us to view every right $kG$-module $M$ as a left $kG$-module $M$. 
Let $M$, $N$ be two $kG$-modules. Then, the tensor product $M\otimes_k N$ is also a $kG$-module with the diagonal action of the group $G$; $g \cdot (x\otimes y) = gx \otimes gy$ for all $g\in G$, $x\in M$, $y\in N$.

\begin{Lemma}\label{rem1}Let $M$, $N$ be two $kG$-modules. Then: 
	\begin{itemize}
		\item [(i)]If $M$ is a flat $kG$-module and the $kG$-module $N$ is $k$-flat, then the $kG$-module $M\otimes_k N$ is also flat. 
		\item[(ii)]If the $kG$-module $N$ is $k$-flat, then $\textrm{fd}_{kG}(M\otimes_k N)\leq \textrm{fd}_{kG}M$. 
	\end{itemize}
	\end{Lemma}

\begin{proof}(i) Since the direct limit of flat modules is flat and the direct limit commutes with the tensor functor, in view of the Govorov-Lazard theorem, it suffices to assume that $M=kG$. Then, $M\otimes_k N=kG\otimes_k N$ which is a flat $kG$-module. 
	
	(ii) It suffices to assume that $\textrm{fd}_{kG}M=n<\infty$. Consider a flat resolution $$0\rightarrow F_n \rightarrow \cdots \rightarrow F_1 \rightarrow F_0 \rightarrow M\rightarrow 0$$ of the $kG$-module $M$. Since $N$ is $k$-flat, we obtain the induced exact sequence $$0\rightarrow F_n\otimes_k N \rightarrow \cdots \rightarrow F_1\otimes_k N \rightarrow F_0\otimes_k N \rightarrow M\otimes_k N\rightarrow 0,$$ which constitutes a $kG$-flat resolution of $ M\otimes_k N$ by (i). Thus, $\textrm{fd}_{kG}(M\otimes_k N)\leq n$.\end{proof}

\begin{Lemma}\label{lemZ}Let $k$ be a commutative ring, $G$ be a group and $H$ be a normal subgroup of $G$. Then, for every flat $k[G/H]$-module $M$ and every $kG$-module $N$ which is flat as $kH$-module, the $kG$-module $M\otimes_k N$ is flat. 
		%\item[(ii)] For every $k[G/H]$-module $M$ and every $kG$-module $N$ which is flat as $kH$-module, we have $\textrm{fd}_{kG}(M\otimes_k N)\leq \textrm{fd}_{k[G/H]} M$.
\end{Lemma}	

\begin{proof}Since the direct limit of flat modules is flat and the direct limit commutes with the tensor functor, by the Govorov-Lazard theorem, it suffices to assume that $M=k[G/H]$. Then, $M\otimes_k N=k[G/H]\otimes_k N\cong \textrm{Ind}^G_H \textrm{Res}^G_H N$ which is a flat $kG$-module (see \cite[Proposition 5.6(a)]{Br}). \end{proof}
	
	%(ii) It suffices to assume that $\textrm{fd}_{k[G/H]} M=n$ is finite. Then, there exist flat $k[G/H]$-modules $F_0, F_1, \dots ,F_n$ and an exact sequence of $k[G/H]$-modules $$0\rightarrow F_n \rightarrow \cdots \rightarrow F_1 \rightarrow F_0 \rightarrow M \rightarrow 0.$$ Since the $kG$-module $N$ is flat as $kH$-module, it follows that $N$ is also $k$-flat and hence we obtain an exact sequence of $kG$-modules (with diagonal action) of the form $$0\rightarrow F_n\otimes_k N \rightarrow \cdots \rightarrow F_1\otimes_k N \rightarrow F_0\otimes_k N \rightarrow M\otimes_k N \rightarrow 0.$$ By (i) we have already prove, the $kG$-modules $F_i\otimes_k N$ are flat, for every $i=0,\dots ,n$, and hence the exact sequence of $kG$-modules above yields $\textrm{fd}_{kG}(M\otimes_k N)\leq n$, as needed.
	
	\begin{Lemma}\label{lem25}Let $k$ be a commutative ring, $G$ be a group and $H$ be a subgroup of $G$.\begin{itemize}\item[(i)] For every Gorenstein flat $kH$-module $M$, the $kG$-module $\textrm{Ind}^G_H M$ is also Gorenstein flat.
	\item[(ii)] For every $kH$-module, $\textrm{Gfd}_{kG}(\textrm{Ind}_H^G M) \leq \textrm{Gfd}_{kH} M$.
%	\item[(iii)] If the subgroup $H$ of $G$ is of finite index, then for every Gorenstein flat $kG$-module, the $kH$-module $\textrm{Res}^G_H M$ is Gorenstein flat as well.
	\end{itemize}  \end{Lemma}
	
	\begin{proof} (i) Let $M$ be a Gorenstein flat $kH$-module. Then, there exists an acyclic complex of flat $kH$-modules $$\textbf{F}=\cdots \rightarrow F_{2}\rightarrow F_1\rightarrow F_0 \rightarrow F_{-1}\rightarrow \cdots$$ such that $M=\textrm{Im}(F_1 \rightarrow F_0)$ and the complex $I\otimes_{kH}\textbf{F}$ is exact whenever $I$ is an injective $kH$-module. Then, the induced complex $$\textrm{Ind}^G_H\textbf{F}=\cdots \rightarrow\textrm{Ind}^G_H F_2 \rightarrow\textrm{Ind}^G_H F_1\rightarrow\textrm{Ind}^G_H F_0 \rightarrow\textrm{Ind}^G_H F_{-1}\rightarrow \cdots$$ is an acyclic complex of flat $kG$-modules and has the $kG$-module $\textrm{Ind}^G_H M$ as syzygy. Moreover, for every injective $kG$-module $I$, the restricted $kH$-module $I$ is also injective. Thus, the isomorphism of complexes $I\otimes_{kG}\textrm{Ind}^G_H \textbf{F} \cong I\otimes_{kH}\textbf{F}$ implies that the $kG$-module $\textrm{Ind}^G_H M$ is also Gorenstein flat.
		
		(ii) It suffices to assume that $\textrm{Gfd}_{kH} M=n$ is finite. Then, there exist Gorenstein flat $kH$-modules $F_0, F_1, \dots ,F_n$ and an exact sequence of $kH$-modules $$0\rightarrow F_n \rightarrow \cdots \rightarrow F_1 \rightarrow F_0 \rightarrow M \rightarrow 0.$$ By (i) which we have already proved, the $kG$-modules $\textrm{Ind}^G_H F_i$ are Gorenstein flat for every $i=0,\dots ,n$. Thus, the induced exact sequence of $kG$-modules $$0\rightarrow\textrm{Ind}^G_H F_n \rightarrow \cdots \rightarrow\textrm{Ind}^G_H F_1 \rightarrow\textrm{Ind}^G_H F_0 \rightarrow\textrm{Ind}^G_H M \rightarrow 0$$ yields $\textrm{Gfd}_{kG}(\textrm{Ind}_H^G M) \leq n$. \end{proof}
		
		%(iii) Let $M$ be a Gorenstein flat $kG$ module. Then, there exists an acyclic complex of flat $kG$-modules $$\textbf{F}=\cdots \rightarrow F_{2}\rightarrow F_1\rightarrow F_0 \rightarrow F_{-1}\rightarrow \cdots,$$ such that $M=\textrm{Im}(F_1 \rightarrow F_0)$ and the complex $J\otimes_{kG}\textbf{F}$ is exact, whenever $J$ is an injective $kG$-module. Then, the restricted complex $$\textrm{Res}^G_H\textbf{F}=\cdots \rightarrow\textrm{Res}^G_H F_2 \rightarrow\textrm{Res}^G_H F_1\rightarrow\textrm{Res}^G_H F_0 \rightarrow\textrm{Res}^G_H F_{-1}\rightarrow \cdots,$$ is an acyclic complex of flat $kH$-modules and has the $kH$-module $\textrm{Res}^G_H M$ as syzygy. Let $I$ be an injective $kH$-module. Then, invoking \cite[Proposition 5.9 III]{Br}, we have the following isomorphisms of complexes $I\otimes_{kH} \textrm{Res}^G_H \textbf{F} \cong \textrm{Ind}^G_H I\otimes_{kG} \textbf{F}\cong \textrm{Coind}^G_H I\otimes_{kG} \textbf{F}$. Since the $kG$-module $\textrm{Coind}^G_H I$ is injective, we conclude that the complex $I\otimes_{kH} \textrm{Res}^G_H \textbf{F}$ is acyclic for every injective $kH$-module $H$, and hence the $kH$-module $\textrm{Res}^G_H M$ is Gorenstein flat, as needed.
	
	The following Lemma will be used several times in the sequel.
	
	\begin{Lemma}\label{lemm22}Let $k$ be a commutative ring, $G$ be a group and $H$ be a subgroup of $G$. Then, $\textrm{sfli}(kH)\leq \textrm{sfli}(kG)$. Moreover, equality holds if the subgroup $H$ of $G$ is of finite index.
	\end{Lemma}
	
	\begin{proof} The proof is the same as the proof of \cite[Proposition 3.10]{Asa} and holds over any commutative ring.\end{proof}	
	
	%It suffices to assume that $\textrm{sfli}(kG)=n<\infty$. Let $I$ be an injective $kH$-module. Since the $kG$-module $\textrm{Coind}^G_H I$ is injective, we have $\textrm{fd}_{kG}\textrm{Coind}^G_H I\leq n$ and hence $\textrm{fd}_{kH}\textrm{Coind}^G_H I\leq n$. As the $kH$-module $I$ is a direct $kH$-summand of $\textrm{Coind}^G_H I$, we obtain that $\textrm{fd}_{kH} I\leq n$. Consequently, we have $\textrm{sfli}(kH)\leq n$.
	
\section{Modules of finite Gorenstein flat dimension} In this section, we prove finiteness criteria for the Gorenstein homological dimension of a group $G$ over a commutative ring $k$ of finite Gorenstein weak global dimension. Moreover, we show that over any commutative ring of finite Gorenstein weak global dimension, the Gorenstein cohomological dimension bounds the Gorenstein homological dimension of a group. A useful tool in our proofs will be the existence of weak characteristic modules for groups.

	\begin{Lemma}\label{lemm29}
	Let $k$ be a commutative ring such that $\textrm{sfli}k<\infty$ and $G$ be a group. Then, every Gorenstein flat $kG$-module is Gorenstein flat as $k$-module and every PGF $kG$-module is PGF as $k$-module.
\end{Lemma}

\begin{proof}
	Let $M$ be a Gorenstein flat $kG$-module. Then, there exists an acyclic complex of flat $kG$-modules $\textbf{F}$ which has as syzygy the $kG$-module $M$ and remains acyclic after the application of the functor $I\otimes_{kG} \_\!\_$, for every injective $kG$-module $I$. The restriction of $\textbf{F}$ is an acyclic complex of flat $k$-modules and the finiteness of $\textrm{sfli}k$ implies that the complex  $I\otimes_{k}\textbf{F}$ is acyclic for every injective $R$-module $I$. Thus, $M$ is Gorenstein flat as $k$-module. The proof of the PGF case is similar.
\end{proof}

The following result was first stated in \cite[Proposition 3.2]{Ren1}.
\begin{Proposition}\label{prop32} Let $k$ be a commutative ring such that $\textrm{sfli}k<\infty$ and consider a group $G$ and a $kG$-module $M$ of finite Gorenstein flat dimension. If the flat dimension $\textrm{fd}_k M$ is finite, then there exists a $k$-pure $kG$-exact sequence $0\rightarrow M \rightarrow A \rightarrow \overline{A} \rightarrow 0$, where $\overline{A}$ is a $k$-flat $kG$-module and $\textrm{Gfd}_{kG}M=\textrm{fd}_{kG}{A}$.
\end{Proposition}

\begin{proof}
	Let $\textrm{Gfd}_{kG}M=n<\infty$. By Proposition \ref{Gfd} there exists a short exact sequence of $kG$-modules $0\rightarrow M \rightarrow A \rightarrow \overline{A}\rightarrow 0$, where $\overline{A}$ is a PGF $kG$-module and $\textrm{fd}_{kG}A= n$. Since the $kG$-module $\overline{A}$ is also Gorenstein flat, invoking Lemma \ref{lemm29} we infer that $\overline{A}$ is a Gorenstein flat $k$-module. Moreover, the finiteness of $\textrm{fd}_{k}M$ and $\textrm{fd}_{k}A$ yields the finiteness of $\textrm{fd}_{k}\overline{A}$. Therefore, $\overline{A}$ is $k$-flat and hence the exact sequence $0\rightarrow M \rightarrow A \rightarrow \overline{A}\rightarrow 0$ is $k$-pure. \end{proof}

\begin{Corollary}\label{prop1}
	Let $k$ be a commutative ring such that $\textrm{sfli}k<\infty$ and $G$ be a group such that $\textrm{Ghd}_{k}G<\infty$. Then, there exists a $k$-pure monomorphism of $kG$-modules $\iota: k \rightarrow A$, where $A$ is a $k$-flat $kG$-module and $\textrm{fd}_{kG}A =\textrm{Ghd}_{k}G$.
\end{Corollary}

\begin{Proposition}\label{propara}Let $k$ be a commutative ring and $G$ be a group such that there exists a $k$-pure monomorphism of $kG$-modules $\iota: k \rightarrow A$, where $A$ is $k$-flat and $\textrm{fd}_{kG}A<\infty$. Then, for every $k$-flat $kG$-module $M$ we have $\textrm{Gfd}_{kG}M \leq \textrm{fd}_{kG}A$.
\end{Proposition}
%Since $M$ is $k$-flat, the restricted exact sequence $\textbf{F}$ is $k$-pure implying the existence of 
\begin{proof} Let $\textrm{fd}_{kG} A =n$ and $M$ be a $k$-flat $kG$-module. We consider a $kG$-flat resolution $$\textbf{F}=\cdots \rightarrow F_2 \rightarrow F_1 \rightarrow F_0 \rightarrow M \rightarrow 0$$ of $M$ and let $M_i=\textrm{Im}(F_i \rightarrow F_{i-1})$, $i\geq 0$, where $M_0=M$, be the corresponding syzygy modules. Since every flat $kG$-module is Gorenstein flat, it suffices to prove that the $kG$-module $M_n$ is Gorenstein flat. Indeed, the $kG$-exact sequence $$0\rightarrow M_n \rightarrow F_{n-1} \rightarrow \cdots \rightarrow F_0 \rightarrow M \rightarrow 0$$ will then be a Gorenstein flat resolution of $M$ of length $n$. 
Since the $kG$-module $A$ is $k$-flat, we obtain an induced $kG$-exact sequence (with diagonal action)$$\textbf{F} \otimes_k A = \cdots \rightarrow F_1\otimes_k A \rightarrow F_0\otimes_k A \rightarrow M\otimes_k A \rightarrow 0.$$
\noindent Then, the exact sequence $\textbf{F} \otimes_k A$ constitutes a flat resolution of the $kG$-module $M\otimes_k A$ and $\textrm{fd}_{kG}(M\otimes_k A)\leq n$ (see Remark \ref{rem1}). Moreover, the corresponding $i$-th syzygies are the modules $(M_i \otimes_k A)_{i\geq 0}$. We obtain that the $kG$-modules $(M_i \otimes_k A)_{i\geq n}$ are flat. A similar argument implies that the diagonal $kG$-module $M_n \otimes_k N \otimes_k A$ is also flat for every $k$-flat $kG$-module $N$. Let $\overline{A}=\textrm{Coker}\iota$ and consider the $k$-pure short exact sequence of $kG$-modules $0\rightarrow k \xrightarrow{\iota} A \rightarrow \overline{A} \rightarrow 0$. Then, for every $j\geq 0$, we obtain a short exact sequence of $RG$-modules of the form 
$$0\rightarrow M_n \otimes_k \overline{A}^{\otimes j} \rightarrow M_n \otimes_k \overline{A}^{\otimes j} \otimes_k A \rightarrow M_n \otimes_k \overline{A}^{\otimes j+1}\rightarrow 0,$$ where we denote by $A^{\otimes j}$ the $j$-th tensor power of $A$ over $k$. Since the $kG$-module $\overline{A}$ is $k$-flat, we have that the diagonal $kG$-modules $M_n \otimes_k \overline{A}^{\otimes j}\otimes_k A$ are flat for every $j\geq 0$ and the splicing of the above short exact sequences yields the exact sequence $$0\rightarrow M_n \xrightarrow{\eta} M_n \otimes_k A \rightarrow M_n \otimes_k \overline{A}\otimes_R A \rightarrow M_n \otimes_k \overline{A}^{\otimes 2}\otimes_k A \rightarrow \cdots .$$ Splicing now the latter exact sequence with the flat $kG$-resolution $$\cdots \rightarrow F_{n+2}\rightarrow F_{n+1}\rightarrow F_n \xrightarrow{\epsilon} M_n \rightarrow 0$$ of $M_n$, we obtain an acyclic complex of flat $kG$-modules
	$$\mathfrak{F}=\cdots \rightarrow F_{n+2}\rightarrow F_{n+1}\rightarrow F_n \xrightarrow{\eta \epsilon}  M_n \otimes_k A \rightarrow M_n \otimes_k \overline{A}\otimes_k A \rightarrow M_n \otimes_k \overline{A}^{\otimes 2}\otimes_k A \rightarrow \cdots ,$$ which has syzygies the $kG$-modules $(M_i)_{i\geq n}$ and $(M_n \otimes_k \overline{A}^{\otimes j})_{j\geq 1}$. In order to prove that the $kG$-module $M_n$ is Gorenstein flat, it suffices to show that the complex $ I\otimes_{kG} \mathfrak{F}$ is acyclic for every injective $kG$-module $I$. Let $I$ be an injective $kG$-module. Then, the $k$-pure short exact sequence of $kG$-modules $0\rightarrow k \xrightarrow{\iota} A \rightarrow \overline{A} \rightarrow 0$ induces the short exact sequence of $kG$-modules (with diagonal action) $0\rightarrow I \rightarrow A\otimes_k I \rightarrow \overline{A}\otimes_k I \rightarrow 0$. The injectivity of the $kG$-module $I$ implies that the latter short exact sequence is $kG$-split. Thus, it suffices to show the acyclicity of the complex $(A\otimes_k I)\otimes_{kG}\mathfrak{F}$. Since the $kG$-module $A$ is $k$-flat, the complex of $kG$-modules (with diagonal action) $\mathfrak{F}\otimes_k A$ is acyclic with syzygies the flat $kG$-modules $(M_i \otimes_k A)_{i\geq n}$ and $(M_n \otimes_k \overline{A}^{\otimes j} \otimes_k A)_{j\geq 1}$. We conclude that the complex $(\mathfrak{F}\otimes_k A)\otimes_{kG}I\cong (A\otimes_k I)\otimes_{kG}\mathfrak{F}$ is acyclic. 
\end{proof}

\begin{Corollary}\label{cor1}Let $k$ be a commutative ring and $G$ be a group such that there exists a $k$-pure monomorphism of $kG$-modules $\iota: k \rightarrow A$, where $A$ is $k$-flat and $\textrm{fd}_{kG}A<\infty$. Then, $\textrm{Ghd}_{k}G \leq \textrm{fd}_{kG}A$.\end{Corollary}

\begin{Definition}\label{defi}Let $k$ be a commutative ring and $G$ be a group. A weak characteristic module for $G$ over $k$ is a $k$-flat $kG$-module $A$ with $\textrm{fd}_{kG}A<\infty$ which admits a $k$-pure $kG$-linear monomorphism $\iota : k \rightarrow A$ (where $k$ is regarded as a trivial $kG$-module).
\end{Definition}

Given Definition \ref{defi}, the $k$-flat $kG$-module $A$ in Corollary \ref{prop1}, Proposition \ref{propara} and Corollary \ref{cor1} is a weak characteristic module for $G$ over $k$.  A weak characteristic module for $G$ over $R$ may not always exist and, if it exists, it is certainly not unique. However, the flat dimension of any weak characteristic module for $G$ over $k$ is uniquely determined by the pair $(k, G)$.

\begin{Corollary}\label{cor37}Let $k$ be a commutative ring and $G$ be a group. Then:
	\begin{itemize}
		\item[(i)]If $A, A'$ are two weak characteristic modules for $G$ over $k$, then $\textrm{fd}_{kG}A =\textrm{fd}_{kG}A'$.
		\item[(ii)]If there exists a weak characteristic module $A$ for $G$ over $k$, then $G$ has finite Gorenstein homological dimension over $k$ and ${{\textrm{Ghd}}_{k}G}\leq \textrm{fd}_{kG}A$.
		\item[(iii)]If $\textrm{sfli}k <\infty$ and ${{\textrm{Ghd}}_{k}G}<\infty$, then there exists a weak characteristic module $A$ for $G$ over $k$ with $\textrm{fd}_{kG} A={{\textrm{Ghd}}_{k}G}$.
	\end{itemize}
\end{Corollary}

\begin{proof}(i) Let $A, A'$ be two weak characteristic modules for $G$ over $k$. Then, $\textrm{fd}_{kG}A <\infty$ and hence $\textrm{Gfd}_{kG}A =\textrm{fd}_{kG}A$ (see \cite[Theorem 2.2]{Ben}). Thus, Proposition \ref{propara} for the weak characteristic module $A'$ and the $k$-flat $kG$-module $A$ yields $\textrm{fd}_{kG}A=\textrm{Gfd}_{kG}A\leq \textrm{fd}_{kG}A'$. Reversing the roles of $A$ and $A'$, we obtain the inequality $\textrm{fd}_{kG}A'\leq \textrm{fd}_{kG}A$. We conclude that $\textrm{fd}_{kG}A=\textrm{fd}_{kG}A'$.
	
	Assertions (ii) and (iii) are precisely Corollaries \ref{cor1} and \ref{prop1}, respectively.
\end{proof}

\begin{Corollary}\label{cor38}Let $k$ be a commutative ring such that $\textrm{sfli}k<\infty$ and $G$ be a group. Then, $G$ has finite Gorenstein homological dimension over $k$ if and only if there exists a weak characteristic module $A$ for $G$ over $k$. In that case, we have $\textrm{fd}_{kG}A={{\textrm{Ghd}}_{k}G}$.
\end{Corollary}

\begin{Remark}\label{remfinal}\rm A characteristic module for $G$ over $k$ is a $k$-projective $kG$-module $A$ with $\textrm{pd}_{kG}A <\infty$ which admits a $k$-split $kG$-linear monomorphism $\iota: k \rightarrow A$. Thus, every characteristic module for $G$ over $k$ is also a weak characteristic module for $G$ over $k$.
\end{Remark} 

\begin{Theorem}\label{prop312}Let $k$ be a commutative ring such that $\textrm{sfli}k<\infty$ and $G$ be a group. Then, $$\textrm{Ghd}_k G \leq \textrm{Gcd}_k G.$$
\end{Theorem}

\begin{proof}It suffices to assume that $\textrm{Gcd}_k G <\infty$. Since $\textrm{sfli}k<\infty$, \cite[Proposition 2.10 (ii)]{St} implies that there exists a characteristic module $A$ for $G$ over $k$ such that $\textrm{Gcd}_k G=\textrm{pd}_{kG}A$. Then, $A$ is also a weak characteristic module for $G$ over $k$ (see Remark \ref{remfinal}) and hence Corollary \ref{cor1} yields the inequality $\textrm{Ghd}_k G \leq \textrm{fd}_{kG}A$. Since $\textrm{fd}_{kG}A \leq \textrm{pd}_{kG}A$, we conclude that $\textrm{Ghd}_k G \leq \textrm{Gcd}_k G.$
\end{proof}

\begin{Corollary}Let $G$ be a group. Then, $\textrm{Ghd}_{\mathbb{Z}} G \leq \textrm{Gcd}_{\mathbb{Z}} G.$
\end{Corollary}

\begin{Remark}\rm  The inequality in Theorem \ref{prop312} may be strict. Indeed, for every commutative ring $k$ such that $\textrm{sfli}k<\infty$ and every infinite locally finite group $G$ we have $\textrm{Ghd}_{k} G=0$ (see \cite[Remark 3.6]{St}), while $\textrm{Gcd}_{k} G>0$ (see \cite[Corollary 2.3]{ET}). It follows that, under the above conditions, the trivial $kG$-module $k$ is Gorenstein flat but not Gorenstein projective.
\end{Remark}

\begin{Proposition}\label{propdion}Let $k$ be a commutative ring and $G$ be a group. Then, $$\textrm{sfli}(kG)\leq \textrm{Ghd}_k G+ \textrm{sfli}k.$$
\end{Proposition}

\begin{proof}It suffices to assume that both ${{\textrm{Ghd}}_{k}G}=n$ and $\textrm{sfli}k=m$ are finite. Then, Proposition \ref{prop1} implies that there exists a $k$-pure monomorphism of $kG$-modules $0\rightarrow k \rightarrow A$, where $A$ is a $k$-flat $kG$-module and $\textrm{fd}_{kG}A =n$. Let $I$ be an injective $kG$-module. It follows that there exists an induced monomorphism of $kG$ modules (with diagonal action) $0\rightarrow I \rightarrow A\otimes_{k}I$. The injectivity of $I$ yields that the latter monomorphism is $kG$-split and hence it suffices to prove that $\textrm{fd}_{kG}(A\otimes_{k}I)\leq n+m$. Indeed, let $\textbf{F}$ be a flat resolution of the $kG$-module $I$ and $I_m=\textrm{Im}(F_m \rightarrow F_{m-1})$ be the corresponding $m$-syzygy. Since $\textrm{sfli}k=m$, we obtain an exact sequence of $kG$-modules $$0\rightarrow I_m \rightarrow F_{m-1} \rightarrow \cdots \rightarrow F_0 \rightarrow I \rightarrow 0,$$ where $I_m,F_{m-1},\dots ,F_0$ are all flat $k$-modules. The flatness of the $k$-module $A$ yields an induced exact sequence of $kG$-modules with diagonal action $$0\rightarrow  A\otimes_k I_m  \rightarrow  A \otimes_k F_{m-1} \rightarrow \cdots \rightarrow  A\otimes_k F_0 \rightarrow A\otimes_k I \rightarrow 0.$$ Since $\textrm{fd}_{kG}(A\otimes_k N)\leq \textrm{fd}_{kG}A$ for every $k$-flat $kG$-module $N$ (see Remark \ref{rem1} (ii)), the exact sequence above implies that $\textrm{fd}_{kG}(A\otimes_{k}I)\leq n+m$, as needed.
\end{proof}

%\begin{Remark}\rm We note that the finiteness of $\textrm{Ghd}_k G$ and $\textrm{sfli}k$ yields that every Gorenstein projective $kG$-module is PGF and hence is Gorenstein flat. Indeed, Proposition \ref{propdion} yields the finiteness of $\textrm{sfli}(kG)$, which means that every complex of flat (or projective) $kG$-modules remains acyclic after applying the functor $I\otimes_{kG} \_\!\_$, for every injective $kG$-module $I$. \end{Remark}

We now restate and prove our second main result (Theorem 1.2).

\begin{Theorem}\label{theo310}Let $k$ be a commutative ring such that $\textrm{sfli}k<\infty$. Then, the following conditions are equivalent for a group $G$:
	\begin{itemize}
		\item[(i)] $\textrm{Gfd}_{kG}M<\infty$ for every $kG$-module $M$.
		\item[(ii)] $\textrm{Ghd}_k G <\infty$.
		\item[(iii)] There exists a short exact sequence of $kG$-modules $0\rightarrow k \rightarrow A \rightarrow \overline{A} \rightarrow 0$, where $\textrm{fd}_{kG}A <\infty$ and the $kG$-module $\overline{A}$ is $PGF$.
		\item[(iv)]There exists a short exact sequence of $kG$-modules $0\rightarrow k \rightarrow A' \rightarrow \overline{A'} \rightarrow 0$, where $\textrm{fd}_{kG}A' <\infty$ and the $kG$-module $\overline{A'}$ is Gorenstein flat.
		\item[(v)] There exists a $k$-split $kG$-monomorphism $\iota: k \rightarrow A$, where $\textrm{fd}_{kG}A <\infty$ and the $kG$-module ${A}$ is $k$-projective.
		\item[(vi)] There exists a $k$-pure $kG$-monomorphism $\iota: k \rightarrow A'$, where $\textrm{fd}_{kG}A' <\infty$ and the $kG$-module ${A'}$ is $k$-flat.
		\item[(vii)] $\textrm{sfli}(kG) <\infty$.
		\end{itemize}
		Moreover, in this case, every Gorenstein projective $kG$-module is Gorenstein flat and hence we have $\textrm{Gfd}_{kG}M\leq\textrm{Gpd}_{kG}M$ for every $kG$-module $M$.
\end{Theorem}

\begin{proof}The implications $(i)\Rightarrow (ii)$, $(iii)\Rightarrow (iv)$ and $(v)\Rightarrow (vi)$ are trivial. Since $kG\cong {(kG)}^{\textrm{op}}$, invoking Theorem \ref{theo012} we obtain the implication $(vii)\Rightarrow (i)$. Moreover, the implication $(ii)\Rightarrow (iii)$ is a direct consequence of Lemma \ref{Gfd}, the implication $(ii)\Rightarrow (vii)$ follows from Proposition \ref{propdion} and the implication $(vi)\Rightarrow (ii)$ follows from Corollary \ref{cor1}.
	
$(iii)\Rightarrow (v):$ Since $\textrm{sfli}k<\infty$, the PGF $kG$-module $\overline{A}$ is PGF as $k$-module (see Lemma \ref{lemm29}). Since the class ${\tt PGF}(k)$ is closed under extensions, we obtain that the $kG$-module $A$ is also PGF as  $k$-module. Moreover, $\textrm{fd}_k A \leq \textrm{fd}_{kG} A<\infty$. Thus, invoking Lemma \ref{lem41} we conclude that the $kG$-module $A$ is $k$-projective. Furthermore, the exact sequence $0\rightarrow k \rightarrow A \rightarrow \overline{A}\rightarrow 0$ is $k$-split, since $\textrm{Ext}^1_k (\overline{A},k)=0$ (see \cite[Theorem 4.11]{S-S}).	

$(iv)\Rightarrow (vi):$ Since $\textrm{sfli}k<\infty$ the Gorenstein flat $kG$-module $\overline{A'}$ is Gorenstein flat as $k$-module (see Lemma \ref{lemm29}). Moreover, we have $\textrm{fd}_k A' \leq \textrm{fd}_{kG} A' <\infty$ and hence the exact sequence $0\rightarrow k \rightarrow A' \rightarrow \overline{A'} \rightarrow 0$ yields $\textrm{fd}_k \overline{A'}<\infty$. Consequently, the $kG$-module $\overline{A'}$ is $k$-flat and the exact sequence $0\rightarrow k \rightarrow A' \rightarrow \overline{A'} \rightarrow 0$ is $k$-pure.

Since $kG\cong {(kG)}^{\textrm{op}}$, the finiteness of $\textrm{sfli}(kG)<\infty$ implies that every syzygy module in a complex of projective $kG$-modules is a PGF $kG$-module. Therefore, every Gorenstein projective $kG$-module is PGF and hence Gorenstein flat. Consequently, we have $\textrm{Gfd}_{kG}M\leq\textrm{Gpd}_{kG}M$ for every $kG$-module $M$.\end{proof}

\begin{Remark} \rm Proposition \ref{propdion} and Theorem \ref{theo310} (i),(ii),(vi) and (vii) are also stated in \cite[Theorem 3.7]{Ren1}. Here, we have provided a different method of proof.
\end{Remark}

\begin{Corollary}\label{corr44}Let $k$ be a commutative ring such that $\textrm{sfli}k<\infty$ and $G$ be a group such that there exists a weak characteristic module for $G$ over $k$. Then, the classes of Gorenstein projective, PGF and Ding projective $kG$-modules all coincide.
\end{Corollary}

\begin{proof}Invoking Theorem \ref{theo310} and its proof, we infer that every Gorenstein projective $kG$-module is PGF. Therefore, using the inclusions ${\tt PGF}(kG)\subseteq {\tt DP}(kG)\subseteq {\tt GProj}(kG)\subseteq{\tt PGF}(kG)$, we conclude that the classes of Gorenstein projective, PGF and Ding projective $kG$-modules all coincide. \end{proof}

\begin{Corollary}Let $k$ be a commutative ring such that $\textrm{sfli}k<\infty$ and $G$ be a group such that there exists a weak characteristic module for $G$ over $k$. Then, every flat Gorenstein projective $kG$-module is projective.
\end{Corollary}

\begin{proof}This follows from Lemma \ref{lem41} and Corollary \ref{corr44}.
\end{proof}

\begin{Corollary}\label{prop45}Let $k$ be a commutative ring such that $\textrm{sfli}k<\infty$ and $G$ be a group such that there exists a weak characteristic module for $G$ over $k$. Then, a $kG$-module $M$ is Gorenstein flat if and only if its Pontryagin dual $DM$ is Gorenstein injective.  
\end{Corollary}

\begin{proof}Since $\textrm{sfli}k<\infty$, invoking Theorem \ref{theo310}, the existence of a weak characteristic module implies that every $kG$-module has finite Gorenstein flat dimension. Therefore, Proposition \ref{Ginj} implies that a $kG$-module $M$ is Gorenstein flat if and only if its Pontryagin dual $DM$ is Gorenstein injective, as needed.  
\end{proof}

\begin{Corollary}Let $k$ be a commutative ring such that $\textrm{sfli}k<\infty$ and $G$ be a group such that there exists a weak characteristic module for $G$ over $k$. Then, $\textrm{Gfd}_{kG}M=\textrm{Gid}_{kG}DM$ for every $kG$-module $M$.
\end{Corollary}

\section{Applications on infinite groups} The goal of this section is to obtain Gorenstein analogues of well known properties of modules in classical homological algebra over the group ring $R=kG$, where $k$ is a commutative ring of finite Gorenstein weak global dimension and $G$ is a group of type $\Phi_k$ or an $\textsc{LH}\mathfrak{F}$ group of type FP$_{\infty}$. 

We recall that a group $G$ is said to be of type FP$_{\infty}$ over $k$ if the trivial $kG$-module $k$ is of type FP$_{\infty}$; i.e. it admits a $kG$-projective resolution $\cdots \rightarrow P_n \rightarrow P_{n-1}\rightarrow \cdots \rightarrow P_0 \rightarrow k \rightarrow 0,$  where $P_i$ is finitely generated and projective for every $i\geq 0$.

The class $\textsc{H}\mathfrak{F}$ was defined by Kropholler in \cite{Kr}. This is the smallest class of groups which contains the class $\mathfrak{F}$ of finite groups and is such that whenever a group $G$ admits a finite dimensional
contractible $G$-CW-complex with stabilizers in $\textsc{H}\mathfrak{F}$, then we also have $G\in \textsc{H}\mathfrak{F}$. The class $\textsc{LH}\mathfrak{F}$ consists of those groups, all of whose finitely generated subgroups are in $\textsc{H}\mathfrak{F}$. All
soluble groups, all groups of finite virtual cohomological dimension and all automorphism
groups of Noetherian modules over a commutative ring are $\textsc{LH}\mathfrak{F}$-groups. The class $\textsc{LH}\mathfrak{F}$ is closed under extensions, ascending unions, free products with amalgamation and HNN extensions.

Let $k$ be a commutative ring. A group $G$ is said to be of type $\Phi_k$ if it has the property that for every $kG$-module $M$, $\textrm{pd}_{kG}M<\infty$ if and only if $\textrm{pd}_{kH}M<\infty$ for every finite subgroup $H$ of $G$. These groups were defined over $\mathbb{Z}$ in \cite{Ta}. Over a commutative ring $k$ of finite global dimension, every group of finite virtual cohomological dimension and every group which acts on a tree with finite stabilizers is of type $\Phi_k$ (see \cite[Corollary 2.6]{MS}).

We consider again a commutative ring $k$ and the $kG$-module $B(G,k)$ which consists of all functions from $G$ to $k$ whose image is a finite subset of $k$. Then, $B(G,k)$ is $k$-free and $kH$-free for every finite subgroup $H$ of $G$. For every $\lambda \in k$, the constant function $\iota(\lambda)\in B(G,k)$ with value $\lambda$ is invariant under the action of $G$. The map $\iota: k \rightarrow B(G,k)$ which is defined in this way is then $kG$-linear and $k$-split. Indeed, for every fixed element $g\in G$, there exists a $k$-linear splitting for $\iota$ by evaluating functions at $g$. Moreover, the cokernel $\overline{B}(G,k)$ of $\iota$ is $k$-free (see \cite[Lemma 3.3]{Kr2} and \cite[Lemma 3.4]{BC}). We note that the $kG$-module $B(G,k)$ is a candidate for a weak characteristic module for any group $G$ over any commutative ring $k$.

\begin{Remark}\rm \label{rem51}Given a commutative ring $k$ and a group $G$, the $kG$-module $B(G,k)$ is weak characteristic if and only if $\textrm{fd}_{kG}B(G,k)<\infty$. It follows from Theorem \ref{theo310} and \cite[Theorem 6.7]{Ste} that for any commutative ring $k$ of finite Gorenstein weak global dimension and any $\textsc{LH}\mathfrak{F}$ group $G$, the flat dimension $\textrm{fd}_{kG}B(G,k)$ is finite if and only if there exists a weak characteristic module for $G$ over $k$.
\end{Remark}

\begin{Remark}\rm \label{rem52}Let $k$ be a commutative ring and $G$ be a group such that $\textrm{pd}_{\mathbb{Z}G}B(G,\mathbb{Z})<\infty$. Then, $\textrm{pd}_{kG}B(G,k)\leq \textrm{pd}_{\mathbb{Z}G}B(G,\mathbb{Z})<\infty$. Indeed, let $\textrm{pd}_{\mathbb{Z}G}B(G,\mathbb{Z})=n<\infty$ and consider a $\mathbb{Z}G$-projective resolution $$0\rightarrow P_n \rightarrow P_{n-1}\rightarrow \cdots \rightarrow P_0 \rightarrow B(G,\mathbb{Z}) \rightarrow 0$$ of $B(G,\mathbb{Z})$. Since $B(G,\mathbb{Z})$ is $\mathbb{Z}$-free, the above exact sequence is $\mathbb{Z}$-split. Thus, we obtain an exact sequence of $kG$-modules $$0\rightarrow P_n\otimes_{\mathbb{Z}}k \rightarrow P_{n-1}\otimes_{\mathbb{Z}}k \rightarrow \cdots \rightarrow  P_0\otimes_{\mathbb{Z}}k \rightarrow B(G,\mathbb{Z})\otimes_{\mathbb{Z}}k =B(G,k) \rightarrow 0$$ which constitutes a $kG$-projective resolution of $B(G,k)$. Hence, $\textrm{pd}_{kG}B(G,k)\leq \textrm{pd}_{\mathbb{Z}G}B(G,\mathbb{Z})$. 
\end{Remark}

\begin{Lemma}\label{lem53}Let $k$ be a commutative ring and $G$ be a group of type $\Phi_k$ or an $\textsc{LH}\mathfrak{F}$ group of type FP$_{\infty}$. Then, the $kG$-module $B(G,k)$ is a weak characteristic module for $G$ over $k$.
\end{Lemma}

\begin{proof}If $G$ is a group of type $\Phi_k$, it follows by the definitions that $\textrm{pd}_{kG}B(G,k)<\infty$. We consider now an $\textsc{LH}\mathfrak{F}$-group of type FP$_{\infty}$. Invoking \cite[Corollary B.2(2)]{Kr2}, which is also valid for $\textsc{LH}\mathfrak{F}$-groups, we infer that $\textrm{pd}_{\mathbb{Z}G} B(G,\mathbb{Z})<\infty$. Then, $\textrm{pd}_{kG}B(G,k)\leq \textrm{pd}_{\mathbb{Z}G}B(G,\mathbb{Z})< \infty$ by Remark \ref{rem52}. Therefore, in each case we have $\textrm{fd}_{kG}B(G,k)\leq\textrm{pd}_{kG}B(G,k)<\infty$ and hence $B(G,k)$ is a weak characteristic module for $G$ over $k$, by Remark \ref{rem51}.
\end{proof}

\begin{Theorem}\label{theo54}Let $k$ be a commutative ring such that $\textrm{sfli}k<\infty$ and $G$ be a group of type $\Phi_k$ or an $\textsc{LH}\mathfrak{F}$ group of type FP$_{\infty}$. Then:
	\begin{itemize}
		\item[(i)]A $kG$-module is Gorenstein projective if and only if it is PGF.
		\item[(ii)]A $kG$-module is Gorenstein flat if and only if its Pontryagin dual module is Gorenstein injective.
	\end{itemize}
\end{Theorem}

\begin{proof}Since $\textrm{sfli}k<\infty$ and $B(G,k)$ is a weak characteristic module for $G$ over $k$ (see Lemma \ref{lem53}), the result follows from Theorem \ref{theo310} and Corollary \ref{prop45}.
\end{proof}

\begin{Corollary}Let $k$ be a commutative ring such that $\textrm{sfli}k<\infty$ and $G$ be a group of type $\Phi_k$ or an $\textsc{LH}\mathfrak{F}$ group of type FP$_{\infty}$. Then, for every $kG$-module $M$ we have:
	\begin{itemize}
		\item[(i)]$\textrm{Gfd}_{kG}M\leq \textrm{PGF-dim}_{kG}M=\textrm{Gpd}_{kG}M$ and
		\item[(ii)]$\textrm{Gfd}_{kG}M=\textrm{Gid}_{kG}DM$.
	\end{itemize}
\end{Corollary}

\begin{Corollary}\label{corr46}Let $k$ be a commutative ring such that $\textrm{sfli}k<\infty$ and $G$ be a group of type $\Phi_k$ or an $\textsc{LH}\mathfrak{F}$ group of type FP$_{\infty}$. Then, every flat Gorenstein projective $kG$-module is projective.
\end{Corollary}

\begin{proof}This follows immediately from Lemma \ref{lem41} and Theorem \ref{theo54}.
\end{proof}

\begin{Remark}\rm In the case where $k$ is a commutative ring of finite Gorenstein weak global dimension and $G$ is a group of type $\Phi_k$ or an $\textsc{LH}\mathfrak{F}$ group of type FP$_{\infty}$, Corollary \ref{corr46} provides a positive answer to \cite[Question 3.8]{Baz}.
\end{Remark}

\section{Some invariants of groups}The goal of this section is to study the relationship between the generalized homological dimension of groups and the Gorenstein homological dimension of groups over any commutative ring of finite Gorenstein weak global dimension. We deduce that the Gorenstein homological dimension ${\textrm{Ghd}}_{\mathbb{Z}}G$ of a group $G$ vanishes if and only if $G$ is locally finite. Moreover, we show that the generalized cohomological dimension bounds the generalized homological dimension over any commutative $\aleph_0$-Noetherian ring of finite global dimension. Finally, we provide a Gorenstein homological analogue of Serre's theorem \cite[VIII, Theorem 3.1]{Br}.

Ikenaga \cite{In} defined the generalized homological dimension $\underline{\textrm{hd}}(G)$ of a group $G$ over $\mathbb{Z}$ as $\underline{\textrm{hd}}(G)=\textrm{sup}\{i: \textrm{Tor}_{i}^{\mathbb{Z}G}(I,M)\neq 0, \, M \,  \mathbb{Z}\textrm{-torsion free},\,I \, \textrm{injective} \}$. Similarly, we may define the generalized homological dimension $\underline{\textrm{hd}}_k(G)$ of a group $G$ over a commutative ring $k$ as $$\underline{\textrm{hd}}_k(G)=\textrm{sup}\{i: \textrm{Tor}_{i}^{kG}(I,M)\neq 0, \, M \,  k\textrm{-flat},\,I \, \textrm{injective} \}.$$ We note that $\underline{\textrm{hd}}(G)=\underline{\textrm{hd}}_{\mathbb{Z}}(G)$. Emmanouil and Talelli \cite{ET2} defined the dimension $\underline{\textrm{h}\textrm{d}}'(G)$ of a group $G$ over $\mathbb{Z}$ as $\underline{\textrm{h}\textrm{d}}'(G)=\textrm{sup}\{i: \textrm{Tor}_{i}^{\mathbb{Z}G}(I,M)\neq 0, \, M \,  \mathbb{Z}\textrm{-free},\,I \, \textrm{injective} \}$ and proved that $\underline{\textrm{h}\textrm{d}}'(G)\leq \underline{\textrm{cd}}(G)$, where $\underline{\textrm{cd}}(G)=\textrm{sup}\{i: \textrm{Ext}^{i}_{\mathbb{Z}G}(M,F)\neq 0, \, M \,  \mathbb{Z}\textrm{-free},\,F\, \textrm{free}\}$ is the generalized cohomological dimension of a group $G$ over $\mathbb{Z}$ (see \cite{In}). Similarly, we may define the dimension $\underline{\textrm{h}\textrm{d}}'_k(G)$ of a group $G$ over a commutative ring $k$ as $$\underline{\textrm{h}\textrm{d}}'_k(G)=\textrm{sup}\{i: \textrm{Tor}_{i}^{kG}(I,M)\neq 0, \, M \,  k\textrm{-projective},\,I \, \textrm{injective} \}.$$ We note that $\underline{\textrm{h}\textrm{d}}'(G)=\underline{\textrm{h}\textrm{d}}'_{\mathbb{Z}}(G)$ and $\underline{\textrm{h}\textrm{d}}'_k(G)\leq \underline{\textrm{h}\textrm{d}}_k(G)$ over any commutative ring $k$.

%\medskip
%The following result was firstly stated in \cite[Proposition 4.6]{Ren1}, but we provide a different proof.

\begin{Proposition}\label{prop316} Let $k$ be a commutative ring and $G$ be a group such that ${\textrm{Ghd}}_{k}G<\infty$. Then, $\textrm{Ghd}_{k}G\leq\underline{\textrm{h}\textrm{d}}\;'_k(G)\leq \underline{\textrm{hd}}_k (G)$. Moreover, if $\textrm{sfli}k<\infty$, then  ${\textrm{Ghd}}_{k}G=\underline{\textrm{h}\textrm{d}}\;'_k(G)=\underline{\textrm{hd}}_k (G)$.
\end{Proposition}

\begin{proof}Since $\textrm{Ghd}_{k}G=\textrm{Gfd}_{kG}k$ is finite and every ring is GF-closed (see \cite[Corollary 4.12]{S-S}), it follows from \cite[Theorem 2.8]{Bennis} that $$\textrm{Ghd}_{k}G=\textrm{sup}\{i:\,\textrm{Tor}_i^{kG}(I,k)\neq 0,\,I \,kG\textrm{-injective}\}.$$ It is now clear that $\textrm{Ghd}_{k}G\leq\underline{\textrm{h}\textrm{d}}'_k(G)\leq \underline{\textrm{hd}}_k (G)$. Moreover, if $\textrm{sfli}k$ is finite, Theorem \ref{theo310} implies that $\textrm{Gfd}_{kG}M<\infty$ for every $kG$-module $M$ and hence \cite[Theorem 2.8]{Bennis} yields $\textrm{Gfd}_{kG}M =\textrm{sup}\{i:\,\textrm{Tor}_i^{kG}(I,M)\neq 0,\,I \,kG\textrm{-injective}\}$. As $\textrm{sfli}k<\infty$ and ${\textrm{Ghd}}_{k}G<\infty$, it follows from Theorem \ref{theo310} that there exists a weak characteristic module $A$ for $G$ such that $\textrm{Ghd}_{k}G=\textrm{fd}_{kG}A$ (see Corollary \ref{cor38}). Then, Proposition \ref{propara} implies that $\textrm{Gfd}_{kG}M \leq \textrm{fd}_{kG}A = \textrm{Ghd}_{k}G$ for every $k$-flat $kG$-module $M$. Consequently, we have $\underline{\textrm{hd}}_k (G)=\textrm{sup}\{\textrm{Gfd}_{kG}M:\, M \, k\textrm{-flat}\}$ $\leq \textrm{Ghd}_{k}G$ and hence $\textrm{Ghd}_k G=\underline{\textrm{hd}}_k (G)$. We conclude that ${\textrm{Ghd}}_{k}G=\underline{\textrm{h}\textrm{d}}'_k(G)=\underline{\textrm{hd}}_k (G)$, as needed.
\end{proof}

\begin{Remark} \rm Proposition \ref{prop316} can also be obtained from \cite[Proposition 4.7]{Ren1} and the definition of $\underline{\textrm{h}\textrm{d}}'_k(G)$. Here, we have offered an alternative proof.
\end{Remark}

Emmanouil and Talelli \cite{ET} proved that over any commutative ring $k$ and any group $G$, the Gorenstein cohomological dimension ${\textrm{Gcd}}_kG$ vanishes if and only if $G$ is a finite group. The following corollary shows that this is not true for the Gorenstein homological dimension.

\begin{Corollary}\label{theo62}Let $G$ be a group. The following conditions are equivalent:
	\begin{itemize}
		\item[(i)] $G$ is locally finite.
		\item[(ii)]$\textrm{Ghd}_k G=0$ for every commutative ring $k$.
		\item[(iii)]$\textrm{Ghd}_{\mathbb{Z}} G=0$.
	\end{itemize}
\end{Corollary}

\begin{proof}The implication $(i)\Rightarrow (ii)$ follows from \cite[Remark 3.6]{St}, while the implication $(ii)\Rightarrow (iii)$ is trivial. Assume now that $\textrm{Ghd}_{\mathbb{Z}}G=0$. Then, Proposition \ref{prop316} implies that $\underline{\textrm{hd}}(G)=0$ as well. Therefore, invoking \cite[Theorem 1]{Em2}, we infer that $G$ is locally finite. We have proved the implication $(iii)\Rightarrow (i)$.
\end{proof}

Another consequence of Proposition \ref{prop316} is given in the following result.

\begin{Corollary}\label{cor314}Let $k$ be a commutative ring and $G$ be a group of finite Gorenstein homological dimension. Then, $$\underline{\textrm{hd}}_k (G) \leq \textrm{sfli}(kG)\leq \underline{\textrm{hd}}_k (G) + \textrm{sfli}k.$$
\end{Corollary}

\begin{proof}For the first inequality, we consider a $kG$-module $M$, an injective $kG$-module $I$ and a nonnegative integer $s$ such that $\textrm{Tor}_s^{kG}(I,M)\neq 0$. Then, $\textrm{fd}_{kG}I\geq s$ and hence $s\leq \textrm{sfli}(kG)$. Consequently, the definition of $\underline{\textrm{hd}}_k (G)$ implies that $\underline{\textrm{hd}}_k (G) \leq \textrm{sfli}(kG)$. 
	
	For the second inequality, we assume that $\underline{\textrm{hd}}_k (G)<\infty$ and $\textrm{sfli}k<\infty$. Then, Proposition \ref{prop316} yields $\textrm{Ghd}_k G=\underline{\textrm{hd}}_k (G)$ and hence the inequality follows from Proposition \ref{propdion}.
\end{proof}

\begin{Remark}\label{rem315}\rm (i) We note that the inequality $\underline{\textrm{hd}}_k (G) \leq \textrm{sfli}(kG)$ in Corollary \ref{cor314} holds for every commutative ring $k$ and every group $G$.

(ii) The inequality in Corollary \ref{cor314} also holds if we replace the condition of the finiteness of the Gorenstein homological dimension of the group $G$ with any of the conditions in Theorem \ref{theo310}.
\end{Remark}

%Let $G$ be a group and $\mu$ be a limit ordinal such that there exists an ascending filtration of $G$ by subgroups $G_{\lambda}$ indexed by ordinals $\lambda \leq \mu$, where $G_{\mu}=G$ and $G_{\lambda}=\bigcup_{\kappa < \lambda} G_{\kappa}$ if $\lambda\leq \mu$ is a limit ordinal. We refer to the ascending chain of subgroups $(G_{\lambda})_{\lambda\leq \mu}$ as an exhaustive continuous ascending filtration of $G$.
%\begin{Corollary}Let $k$ be a commutative ring and consider a group $G$, a limit ordinal $\mu$ and an exhaustive continuous ascending filtration of $G$ by subgroups $(G_{\lambda})_{\lambda \leq \mu}$. Then, $$\textrm{Ghd}_k (G) \leq \textrm{sup}_{\lambda<\mu}\textrm{Ghd}_k (G_{\lambda}).$$\end{Corollary}

%\begin{proof}It suffices to assume that $\textrm{sup}_{\lambda<\mu}\textrm{Ghd}_k (G_{\lambda})=n<\infty$. In that case, we have $\textrm{Ghd}_k (G_{\lambda})\leq n$, for every $\lambda<\mu$.\end{proof}

Similarly, we define $\underline{\textrm{cd}}_k(G)=\textrm{sup}\{i: \textrm{Ext}^{i}_{kG}(M,P)\neq 0, \, M \,  k\textrm{-projective},\,P\, \textrm{projective}\}$ and note that $\underline{\textrm{cd}}_\mathbb{Z}(G)=\underline{\textrm{cd}}(G)$ is the generalized cohomological dimension of a group $G$ over $\mathbb{Z}$ defined in \cite{In}.

\begin{Proposition}\label{propp65}Let $k$ be a commutative ring and $G$ be a group. Then, $$\underline{\textrm{cd}}_k (G)\leq \textrm{silp}(kG) \leq \underline{\textrm{cd}}_k (G) +\textrm{gldim}k.$$
\end{Proposition}

\begin{proof}The first inequality follows from \cite[Lemma 2.16]{St}. For the second inequality, it suffices to assume that $\underline{\textrm{cd}}_k (G)=m$ and $\textrm{gldim}k=n$ are finite. Let $N$ be a $kG$-module and consider a $kG$-projective resolution $$\textbf{P}=\cdots \rightarrow P_n \rightarrow \cdots \rightarrow P_1 \rightarrow P_0 \rightarrow N \rightarrow 0 $$ of $N$. Since $\textrm{gldim}k=n$, viewing \textbf{P} as a $k$-projective resolution of $N$, we obtain an exact sequence of $kG$-modules $$0\rightarrow M \rightarrow P_{n-1}\rightarrow \cdots \rightarrow P_1 \rightarrow P_0 \rightarrow N \rightarrow 0,$$ where the $kG$-modules $P_i$ are projective for every $i=0,1,\dots,n-1$ and $M$ is $k$-projective. Let $P$ be a $kG$-projective module. Then, the hypothesis $\underline{\textrm{cd}}_k (G)=m$ implies that $\textrm{Ext}^{i}_{kG}(M,P)=0$ for every $i>m$. Hence $\textrm{Ext}^i_{kG}(N,P)\cong \textrm{Ext}^{i-n}_{kG}(M,P)=0$ for every $i>m+n$. We conclude that $\textrm{id}_{kG}P\leq n+m$ for every $kG$-projective module $P$. Hence $\textrm{silp}(kG)\leq m+n$, as needed.  
\end{proof}

We recall that a left (resp. right) $\aleph_0$-Noetherian ring is a ring all of whose left (resp. right) ideals are countably generated. Countable rings and countably generated algebras over fields are examples of rings which are both left and right $\aleph_0$-Noetherian.

\begin{Proposition}\label{propp66}Let $k$ be a commutative $\aleph_0$-Noetherian ring and $G$ be a group. Then, $\textrm{silp}(kG)=\textrm{spli}(kG)$.
\end{Proposition}

\begin{proof}Invoking \cite[Proposition 4.3]{em}, we infer that $\textrm{spli}(kG)\leq\textrm{silp}(kG)$. Since $kG\cong {(kG)}^{\textrm{op}}$, \cite[Corollary 24]{DE} implies that $\textrm{silp}(kG)\leq\textrm{spli}(kG)$. We conclude that $\textrm{silp}(kG)=\textrm{spli}(kG)$, as needed.
\end{proof}

\begin{Corollary}\label{cor67}Let $k$ be a commutative $\aleph_0$-Noetherian ring and $G$ be a group. Then,
	$$\underline{\textrm{cd}}_k (G)\leq \textrm{spli}(kG) \leq \underline{\textrm{cd}}_k (G) +\textrm{gldim}k.$$
\end{Corollary}

\begin{proof}This follows from Propositions \ref{propp65} and \ref{propp66}.
\end{proof}
The following result yields a generalization of \cite[Theorem 2.5]{BDT}.  We denote by ${\widetilde{\textrm{Gcd}}_{k}G}$ the projectively coresolved Gorenstein flat dimension of a group $G$ over a commutative ring $k$ (see \cite{St}).

\begin{Proposition}\label{prop68}Let $k$ be a commutative $\aleph_0$-Noetherian ring of finite global dimension and $G$ be a group. Then, $\textrm{Gcd}_kG={\widetilde{\textrm{Gcd}}_{k}G}=\underline{\textrm{cd}}_k(G)$. 
\end{Proposition}

\begin{proof}The first equality is exactly \cite[Proposition 2.13]{St}. For the second equality, we distinguish the cases ${\widetilde{\textrm{Gcd}}_{k}G}<\infty$ and ${\widetilde{\textrm{Gcd}}_{k}G}=\infty$. In the first case, the equality follows from \cite[Proposition 2.15]{St}. In the second case, invoking \cite[Theorem 2.14]{St}, we infer that $\textrm{spli}(kG)=\infty$ and hence $\underline{\textrm{cd}}_k(G)=\infty$ by Corollary \ref{cor67}, as needed.
\end{proof}
	%By Corollary \ref{cor67}, we obtain that $\underline{\textrm{cd}}_k(G)$ is finite if and only if $\textrm{spli}(kG)$ is finite. Invoking \cite[Theorem 2.14]{St}, we infer that $\textrm{spli}(kG)$ is finite if and only if $\textrm{Gcd}_kG$ is finite. Consequently, the dimension $\underline{\textrm{cd}}_k(G)$ is finite if and only if the Gorenstein cohomological dimension $\textrm{Gcd}_kG$ of $G$ is finite. Thus, we may assume that both dimensions $\underline{\textrm{cd}}_k(G)$ and $\textrm{Gcd}_kG$ are finite. Since the commutative ring $k$ is of finite global dimension, we have ${\textrm{Gcd}}_{k}G={\widetilde{\textrm{Gcd}}_{k}G}$ by \cite[Proposition 2.13]{St}. Therefore, invoking \cite[Proposition 2.15]{St}, we conclude that $\textrm{Gcd}_kG={\widetilde{\textrm{Gcd}}_{k}G}=\underline{\textrm{cd}}_k(G)$, as needed.

The following result describes the relationship between the generalized homological dimension and the generalized cohomological dimension of a group $G$, defined by Ikenaga \cite{In}.

\begin{Theorem}\label{propp16}Let $k$ be a commutative $\aleph_0$-Noetherian ring of finite global dimension and $G$ be a group. Then, $\underline{\textrm{hd}}\;'_k(G)\leq \underline{\textrm{hd}}_k(G)\leq \underline{\textrm{cd}}_k(G)$.
\end{Theorem}

\begin{proof}It suffices to prove the inequality $\underline{\textrm{hd}}_k(G)\leq \underline{\textrm{cd}}_k(G)$. Invoking Proposition \ref{prop68}, we infer that $\textrm{Gcd}_{k}G=\underline{\textrm{cd}}_k(G)$ and hence it suffices to show that $\underline{\textrm{hd}}_k(G)\leq \textrm{Gcd}_{k}G$. We may assume that $\textrm{Gcd}_{k}G$ is finite. Then, Theorem \ref{prop312} implies that $\textrm{Ghd}_{k}G\leq \textrm{Gcd}_{k}G$ and hence $\textrm{Ghd}_{k}G$ is also finite. Therefore, we have ${\textrm{Ghd}}_{k}G=\underline{\textrm{hd}}_k(G)$ by Proposition \ref{prop316}. We conclude that $\underline{\textrm{hd}}_k(G)\leq \textrm{Gcd}_{k}G$, as needed.
\end{proof}

\begin{Corollary}\label{prop16} Let $G$ be a group. Then, $\underline{\textrm{hd}}\;'(G)\leq \underline{\textrm{hd}}(G)\leq \underline{\textrm{cd}}(G)$.
\end{Corollary}

Serre's Theorem \cite[VIII Theorem 3.1]{Br} yields an equality between cohomological dimensions of a group and subgroups with finite index. In the next result we give an analogous result concerning Gorenstein homological dimensions of groups. 

%For the proof of our analogue we will make use of the following result which was also stated in \cite[Proposition 4.1]{Ren1}.\begin{Proposition}\label{prop66} Let $k$ be a commutative ring such that $\textrm{sfli}k<\infty$ and consider a group $G$ and a subgroup $H$ of $G$. Then $\textrm{Ghd}_kH\leq \textrm{Ghd}_kG$.\end{Proposition}	\begin{proof}It suffices to assume that ${\textrm{Ghd}}_{k}G=n$ is finite. It follows from Corollary \ref{prop1} that there exists a $k$-pure monomorphism of $RG$-modules $\iota: k \rightarrow A$, where $A$ is a $k$-flat $kG$-module and $\textrm{fd}_{kG}A =n$. We note that the $k$-pure monomorphism $\iota$ is restricted to a $k$-pure monomorphism of $kH$-modules $\iota|_H: k \rightarrow A$. Since $\textrm{fd}_{kH}A\leq \textrm{fd}_{kG}A= n$, Corollary \ref{cor1} yields ${\textrm{Ghd}}_k H\leq \textrm{fd}_{kH}A \leq n$, as needed.\end{proof}

\begin{Theorem}\label{theo612}Let $k$ be a commutative ring such that $\textrm{sfli}k<\infty$ and consider a group $G$ and a subgroup $H$ of $G$ of finite index. Then, $\textrm{Ghd}_{k}G=\textrm{Ghd}_{k}H$.
\end{Theorem}

\begin{proof}Since the subgroup $H$ of $G$ is of finite index, Lemma \ref{lem22} yields $\textrm{sfli}(kG)=\textrm{sfli}(kH)$. Moreover, \cite[Proposition 4.1]{Ren1} yields ${\textrm{Ghd}_{k}H}\leq {\textrm{Ghd}_{k}G}$ and hence it suffices to show that ${\textrm{Ghd}_{k}G}\leq {\textrm{Ghd}_{k}H}$. For that, it suffices to assume that ${\textrm{Ghd}_{k}H}$ is finite. Since $\textrm{sfli}k<\infty$, applying Theorem \ref{theo310}, we deduce that $\textrm{sfli}(kH)<\infty$ and hence $\textrm{sfli}(kG)<\infty$ as well. Using again Theorem \ref{theo310}, we infer that ${\textrm{Ghd}}_{k}G<\infty$. Therefore, Proposition \ref{prop316} yields $\textrm{Ghd}_{k}(G)=\underline{\textrm{hd}}_k (G)$ and $\textrm{Ghd}_{k}H=\underline{\textrm{hd}}_k (H)$. Consequently, it suffices to show that $\underline{\textrm{hd}}_k (G)=\underline{\textrm{hd}}_k (H)$. Applying \cite[Proposition 11]{In}, which holds over any commutative ring $k$, we conclude that  $\textrm{Ghd}_{k}G=\textrm{Ghd}_{k}H$.\end{proof}

\begin{Corollary}Let $G$ be a group and $H$ be a subgroup of $G$ of finite index. Then, $\textrm{Ghd}_{\mathbb{Z}}G=\textrm{Ghd}_{\mathbb{Z}}H$.
\end{Corollary}

Another property of the Gorenstein homological dimension of groups related to subgroups is given in the next proposition.

\begin{Proposition}Let $k$ be a commutative ring and $G$ be a group such that $G={\lim\limits_{\longrightarrow}}_{\alpha} G_{\alpha}$, where $\{G_{\alpha}\}$ is a family of subgroups of $G$. Then, ${\textrm{Ghd}_k G \leq \textrm{sup}_{\alpha}\textrm{Ghd}_k G_{\alpha}}$.
\end{Proposition}

\begin{proof}It suffices to assume that $\textrm{sup}_{\alpha}\textrm{Ghd}_k G_{\alpha}=n<\infty.$ Then, we have $\textrm{Ghd}_k G_{\alpha}\leq n$ for every $\alpha$.
	It follows from Lemma \ref{lem25} (ii) that $\textrm{Gfd}_{kG}(\textrm{Ind}_{G_{\alpha}}^G k) \leq \textrm{Gfd}_{kG_{\alpha}} k=\textrm{Ghd}_k G_{\alpha}\leq n$. Since the class $\overline{{\tt GFlat}}_n(kG)$ is closed under direct limits (see Lemma \ref{lemX}), we obtain that $\textrm{Ghd}_k G=\textrm{Gfd}_{kG}k=\textrm{Gfd}_{kG}({\lim\limits_{\longrightarrow}}_{\alpha}\textrm{Ind}^G_{G_{\alpha}}k)\leq n$, as needed.
\end{proof}		

\section{Bounds for the Gorenstein weak global dimension of group rings} The goal of this section is to construct estimates for the invariant $\textrm{sfli}(kG)$ over any commutative ring $k$ and any group $G$. As a result, we obtain estimates for the Gorenstein weak global dimension $\textrm{Gwgl.dim}(kG)$. Moreover, we generalize \cite[Proposition 4.6]{ET2} and \cite[Theorem 3.11]{Asa} by proving the subadditivity of $\textrm{sfli}(kG)$ under group extensions over any commutative ring $k$.

\begin{Proposition}\label{prop61}Let $k$ be a commutative ring and $G$ be a group. Then, $$\textrm{Ghd}_{k}G\leq \textrm{sfli}(kG).$$
\end{Proposition}

\begin{proof}It suffices to assume that $\textrm{sfli}(kG)$ is finite. Then, $\textrm{sfli}k<\infty$ by Lemma \ref{lemm22} for $H=\{1\}$ and hence Theorem \ref{theo310} yields $\textrm{Ghd}_{k}G<\infty$. Thus, invoking Proposition \ref{prop316}, we infer that $\textrm{Ghd}_{k}G=\underline{\textrm{hd}}_k(G)$. Since $\underline{\textrm{hd}}_k (G) \leq \textrm{sfli}(kG)$ (see Remark \ref{rem315}(i)), we conclude that $\textrm{Ghd}_{k}G\leq \textrm{sfli}(kG)$, as needed.
\end{proof}

\begin{Corollary}\label{cor57} Let $k$ be a commutative ring and $G$ be a group. Then, $$\textrm{max}\{\textrm{Ghd}_{k}G,\textrm{sfli}k\}\leq \textrm{sfli}(kG)\leq \textrm{Ghd}_{k}G+\textrm{sfli}k.$$
\end{Corollary}

\begin{proof}The first inequality $\textrm{max}\{\textrm{Ghd}_{k}G,\textrm{sfli}k\}\leq \textrm{sfli}(kG)$ follows from Proposition \ref{prop61} and Lemma \ref{lemm22} for $H=\{1\}$. The second inequality is exactly Proposition \ref{propdion}.\end{proof}

\begin{Corollary}Let $k$ be a commutative ring and $G$ be a locally finite group. Then, $\textrm{sfli}(kG)=\textrm{sfli}k$.
\end{Corollary}

\begin{proof}This follows from Corollary \ref{cor57} and Theorem \ref{theo62}, since $\textrm{Ghd}_{k}G=0$ for every locally finite group $G$. \end{proof}

\begin{Corollary}\label{cor58} Let $k$ be a commutative ring and $G$ be a group. Then, $$\textrm{max}\{\textrm{Ghd}_{k}G,\textrm{Gwgl.dim}k\}\leq \textrm{Gwgl.dim}(kG)\leq \textrm{Ghd}_{k}G+\textrm{Gwgl.dim}k.$$
\end{Corollary}

\begin{proof}Since $kG\cong {(kG)}^{\textrm{op}}$, this is a direct consequence of Corollary \ref{cor57} and Theorem \ref{theo012}.\end{proof}

\begin{Proposition}\label{prop65}Let $k$ be a commutative ring and consider a group $G$, a normal subgroup $H$ of $G$ and the corresponding quotient group $Q=G/H$. Then, $$\textrm{sfli}(kG) \leq \textrm{sfli}(kH) + \textrm{Ghd}_k Q \leq \textrm{sfli}(kH) + \textrm{sfli}(kQ).$$
\end{Proposition}

\begin{proof}In view of Proposition \ref{prop61}, it suffices to prove the first inequality. The quotient homomorphism $G\rightarrow Q$ enables us to regard every $kQ$-module as an $kG$-module and every $kQ$-linear map as an $kG$-linear map. It suffices to assume that $\textrm{sfli}(kH)=n$ and $\textrm{Ghd}_k Q =m$ are finite. Then, $\textrm{sfli}k<\infty$ by Lemma \ref{lemm22} and hence Theorem \ref{theo310} yields the existence of a $k$-pure monomorphism of $kQ$-modules $\iota: k \rightarrow A$, where $A$ is a $k$-flat $kQ$-module and $\textrm{fd}_{kQ}A =m<\infty$ (see Corollary \ref{cor38}). It follows that there exists an exact sequence of $kQ$-modules:
\begin{equation*}\label{eq1}\textbf{Q}= 0\rightarrow F_{m} \rightarrow \cdots \rightarrow F_1\rightarrow F_0 \rightarrow A \rightarrow 0,
\end{equation*}
where the $kQ$-module $F_i$ is flat for every $i=0,1,\dots,m$. We consider now an injective $kG$-module $I$ and a truncated $kG$-flat resolution of $I$ of length $n$:
\begin{equation*}\label{eq5}
\textbf{F}=	0\rightarrow I_n \rightarrow F'_{n-1} \rightarrow \cdots \rightarrow F'_0 \rightarrow I \rightarrow 0.
\end{equation*}
Since $\textrm{sfli}(kH)=n$, the $kH$-modules $I_n,F'_{n-1},\dots ,F'_0$ are flat. Since $A$ is $k$-flat, it follows from Künneth's formula that the total complex of the double complex $\textbf{Q}\otimes_k \textbf{F}$ yields a resolution of $A\otimes_k I$ by $kG$-modules of length $n+m$. Applying Lemma \ref{lemZ}, we infer that this resolution consists of flat $kG$-modules and hence $\textrm{fd}_{kG}(A\otimes_k I )\leq n+m$. The $k$-pure monomorphism of $kG$-modules $\iota: k \rightarrow A$ yields a $kG$-monomorphism $I=k\otimes_k I \rightarrow A\otimes_k I$ which is $kG$-split, and hence $\textrm{fd}_{kG}I\leq \textrm{fd}_{kG}(A\otimes_k I )\leq n+m$. It follows that $\textrm{sfli}(kG)\leq n+m$, as needed.\end{proof}

The following result proves that the invariant $\textrm{silf}(kG)$ is also subadditive under group extensions over any commutative $\aleph_0$-Noetherian ring. This generalizes \cite[Theorem 3.11(ii)]{Asa}.

\begin{Proposition}\label{proppp65}Let $k$ be a commutative $\aleph_0$-Noetherian ring and consider a group $G$, a normal subgroup $H$ of $G$ and the corresponding quotient group $Q=G/H$. Then, $$\textrm{silf}(kG) \leq \textrm{silf}(kH) + \textrm{Gcd}_k Q \leq \textrm{silf}(kH) + \textrm{silf}(kQ).$$
\end{Proposition}

\begin{proof}Since $k$ is $\aleph_0$-Noetherian, it follows from Proposition \ref{propp66} and \cite[Proposition 2.1]{ET2} that $\textrm{silf}(kG)=\textrm{silp}(kG)=\textrm{spli}(kG)$. Applying \cite[Proposition 2.21]{St} which holds over any commutative ring, \cite[Proposition 2.13]{St} and \cite[Corollary 2.18]{St}, we deduce that $\textrm{spli}(kG) \leq \textrm{spli}(kH) + \textrm{Gcd}_{k}Q \leq \textrm{spli}(kH) + \textrm{spli}(kQ).$  It follows that $\textrm{silf}(kG) \leq \textrm{silf}(kH) + \textrm{Gcd}_k Q \leq \textrm{silf}(kH) + \textrm{silf}(kQ),$ as needed.\end{proof}

\begin{Corollary}\label{cor76}Let $k$ be a commutative ring and consider a group $G$, a normal subgroup $H$ of $G$ and the corresponding quotient group $G/H$. Then, $$\textrm{Gwgl.dim}(kG) \leq \textrm{Gwgl.dim}(kH) + \textrm{Ghd}_k Q\leq \textrm{Gwgl.dim}(kH) + \textrm{Gwgl.dim}(kQ).$$
\end{Corollary}

\begin{proof}Since $kG\cong {(kG)}^{\textrm{op}}$, this is a direct consequence of Proposition \ref{prop65} and Theorem \ref{theo012}.
\end{proof}

\textit{Acknowledgement.} The authors wish to thank the anonymous referee for reading the paper carefully and providing useful comments and suggestions.

%\section*{Declarations}
%\noindent\textbf{Conflict of interest.} The authors declare that they have no conflict of interest.

%\noindent\textbf{Data availability.} No data was used for the preparation of this manuscript.

%\noindent\textbf{Funding information.} No funding was received to assist with the preparation of this manuscript.

\bigskip
{\small {\sc Department of Mathematics,
            National and Kapodistrian University of Athens,
             Athens 15784,
             Greece}}

{\em E-mail address:} {\tt kaperonn@math.uoa.gr}

\bigskip
{\small {\sc School of Applied Mathematical and Physical Sciences, National Technical 
		University of Athens,
		Athens 15780,
		Greece}}
	
{\em E-mail address:} {\tt dstergiop@math.uoa.gr}

\end{document}